\documentclass[a4paper,11pt]{amsart}
\usepackage[mathscr]{eucal}
\usepackage{fancybox}
\usepackage{amscd}
\usepackage{amsmath}
\usepackage{amssymb}
\usepackage{amsthm}
\usepackage{float}
\usepackage[pdftex]{graphicx}
\usepackage{tikz}
\usepackage[all, ps, dvips, pdftex]{xy}
\usepackage{color}
\usepackage[hidelinks, hyperfootnotes=false]{hyperref}
\usepackage{array}
\usepackage{amscd}

\DeclareFontFamily{U}{rsfs}{%
\skewchar\font127}
\DeclareFontShape{U}{rsfs}{m}{n}{%
<-6>rsfs5<6-8.5>rsfs7<8.5->rsfs10}{}
\DeclareSymbolFont{rsfs}{U}{rsfs}{m}{n}
\DeclareSymbolFontAlphabet
{\mathrsfs}{rsfs}
\DeclareRobustCommand*\rsfs{%
\@fontswitch\relax\mathrsfs}

\theoremstyle{plain}
\newtheorem{thm}{Theorem}[section]
\newtheorem*{thm*}{Theorem}
\newtheorem{prop}[thm]{Proposition}

\newtheorem{lem}[thm]{Lemma}

\newtheorem{defi}[thm]{Definition}
\newtheorem{rmk}[thm]{Remark}

\newtheorem{prop-defi}[thm]{Proposition-Definition}
\newtheorem{thm-defi}[thm]{Theorem-Definition}
\newtheorem{defi-thm}[thm]{Definition-Theorem}
\newtheorem{lem-defi}[thm]{Lemma-Definition}

\newtheorem*{question*}{Question}
\newtheorem{assum}[thm]{Assumption}

\newtheorem{setup-def}[thm]{Setup-Definition}

\newdimen\argwidth
\def\db[#1\db]{
 \setbox0=\hbox{$#1$}\argwidth=\wd0
 \setbox0=\hbox{$\left[\box0\right]$}
  \advance\argwidth by -\wd0
 \left[\kern.3\argwidth\box0 \kern.3\argwidth\right]}

\newcommand{\aA}{\mathcal{A}}
\newcommand{\bB}{\mathcal{B}}
\newcommand{\cC}{\mathcal{C}}

\newcommand{\eE}{\mathcal{E}}
\newcommand{\fF}{\mathcal{F}}
\newcommand{\gG}{\mathcal{G}}
\newcommand{\hH}{\mathcal{H}}
\newcommand{\iI}{\mathcal{I}}

\newcommand{\kK}{\mathcal{K}}

\newcommand{\mM}{\mathcal{M}}
\newcommand{\nN}{\mathcal{N}}
\newcommand{\oO}{\mathscr{O}}
\newcommand{\pP}{\mathcal{P}}

\newcommand{\vV}{\mathcal{V}}
\newcommand{\wW}{\mathcal{W}}

\newcommand{\zZ}{\mathcal{Z}}

\newcommand{\Ob}{\mathcal{O}b}

\newcommand{\bL}{\mathbb{L}}
\newcommand{\bQ}{\mathbb{Q}}

\newcommand{\fm}{\mathfrak{m}}

\newcommand{\fc}{\mathfrak{c}}
\newcommand{\fn}{\mathfrak{n}}

\renewcommand{\tilde}{\widetilde}

\newcommand{\tU}{\tilde{U}}

\newcommand{\coker}{\mathop{\rm coker}\nolimits}

\newcommand{\lr}{\longrightarrow}

\newcommand{\Hom}{\mathop{\rm Hom}\nolimits}

\newcommand{\id}{\textrm{id}}

\newcommand{\Spec}{\mathop{\rm Spec}\nolimits}

\newcommand{\Coh}{\mathop{\rm Coh}\nolimits}

\newcommand{\Sym}{\mathop{\rm Sym}\nolimits}

\newcommand{\bC}{\mathbb{C}}

\newcommand{\bP}{\mathbb{P}}

\def\lal{_\lambda}
\def\vir{\mathrm{\vir}}
\def\loc{\mathrm{\loc}}

\def\sun{\underline{\sigma}}
\def\labc{_{\alpha \beta \gamma}}

\def\tE{\tilde{E}}

\def\lra{\longrightarrow}
\def\ti{\tilde}

\def\sO{\mathscr O}
\def\CC{\mathbb C}

\def\lalp{_\alpha}
\def\sub{\subset}

\def\virt{^{\mathrm{vir}}}

\def\beq{\begin{equation}}
\def\eeq{\end{equation}}

\def\lalp{_\alpha }
\def\tE{\tilde{E}}
\def\tphi{\tilde{\phi}}
\def\lab{_{\alpha\beta}}
\def\lbet{_\beta}
\def\mapright#1{\,\smash{\mathop{\lra}\limits^{#1}}\,}
\def\mapleft#1{\,\smash{\mathop{\longleftarrow}\limits^{#1}}\,}
\def\rred{{\mathrm{red}}}
\def\loc{{\mathrm{loc}}}

\makeatletter
 
  \@addtoreset{equation}{section}
\makeatother

\makeatletter
\def\@tocline#1#2#3#4#5#6#7{\relax
  \ifnum #1>\c@tocdepth 
  \else
    \par \addpenalty\@secpenalty\addvspace{#2}%
    \begingroup \hyphenpenalty\@M
    \@ifempty{#4}{%
      \@tempdima\csname r@tocindent\number#1\endcsname\relax
    }{%
      \@tempdima#4\relax
    }%
    \parindent\z@ \leftskip#3\relax \advance\leftskip\@tempdima\relax
    \rightskip\@pnumwidth plus4em \parfillskip-\@pnumwidth
    #5\leavevmode\hskip-\@tempdima
      \ifcase #1
       \or\or \hskip 1em \or \hskip 2em \else \hskip 3em \fi%
      #6\nobreak\relax
    \hfill\hbox to\@pnumwidth{\@tocpagenum{#7}}\par
    \nobreak
    \endgroup
  \fi}
\makeatother

\title[Localization Formulas for APOT]{Localizing Virtual Structure Sheaves for Almost Perfect Obstruction Theories}

\author{Young-Hoon Kiem}
\address{Department of Mathematical Sciences and Research Institute of Mathematics, Seoul National University, Seoul 08826, Korea}
\email{kiem@snu.ac.kr}

\author{Michail Savvas}
\address{Department of Mathematics, University of California, San Diego, La Jolla, CA 92093, USA}
\email{msavvas@ucsd.edu}

\thanks{YHK was partially supported by Samsung Science and Technology Foundation grant SSTF-BA1601-01.}

\begin{document}

\maketitle

\begin{abstract} 
Almost perfect obstruction theories were introduced in an earlier paper by the authors as the appropriate notion in order to define virtual structure sheaves and $K$-theoretic invariants for many moduli stacks of interest, including $K$-theoretic Donaldson-Thomas invariants of sheaves and complexes on Calabi-Yau threefolds. The construction of virtual structure sheaves is based on the $K$-theory and Gysin maps of sheaf stacks.

In this paper, we generalize the virtual torus localization and cosection localization formulas and their combination to the setting of almost perfect obstruction theory. To this end, we further investigate the $K$-theory of sheaf stacks and its functoriality properties. As applications of the localization formulas, we establish a $K$-theoretic wall crossing formula for simple $\bC^\ast$-wall crossings and define $K$-theoretic invariants refining the Jiang-Thomas virtual signed Euler characteristics.
\end{abstract}

\tableofcontents

\def\cG{{\mathcal{G}}}
\def\cF{{\mathcal{F}}}
\def\cK{{\mathcal{K}}}
\def\DM{Deligne-Mumford }
\def\sO{{\mathscr O}}
\def\Coh{\mathrm{Coh} }
\def\coh{\underline{\mathrm{Coh}} }
\def\QQ{\mathbb{Q} }
\def\PP{\mathbb{P} }
\def\AA{\mathbb{A} }

\section{Introduction}

Enumerative geometry is the study of counts of geometric objects subject to a set of given conditions. More often than not, the moduli stacks parameterizing the objects of interest are highly singular, have many components of dimension different from the expected dimension and do not behave well under deformation.

To address these issues, Li-Tian \cite{LiTian} and Behrend-Fantechi \cite{BehFan} developed the theory of virtual fundamental cycles, which have been instrumental in defining and investigating several algebro-geometric enumerative invariants of great importance, such as Gromov-Witten \cite{BehGW}, Donaldson-Thomas \cite{Thomas} and Pandharipande-Thomas \cite{PT1} invariants, and are still one of the major components in modern enumerative geometry.

Any Deligne-Mumford stack $X$ is equipped with its intrinsic normal cone $\cC_X$ which locally for an \'{e}tale map $U \to X$ and a closed embedding $U \hookrightarrow V$ into a smooth scheme $V$ is the quotient stack $[C_{U/V} / T_V|_U]$ (cf. \cite{BehFan}). A perfect obstruction theory $\phi \colon E \to \bL_X$ gives an embedding of $\cC_X$ into the vector bundle stack $\eE_X = h^1 / h^0 ( E^\vee )$ and the virtual fundamental cycle is defined as the intersection of the zero section $0_{\eE_X}$ with $\cC_X$
\begin{align*}
[X]\virt := 0_{\eE_X}^! [\cC_X] \in A_{\mathrm{v. dim.}} (X).
\end{align*}

Recently, a lot of interest has been generated towards refinements of enumerative invariants that go beyond numbers or the intersection theory of cycles. Motivated by theoretical physics and geometric representation theory, it is in particular desirable to obtain such a refinement in $K$-theory (cf. for example \cite{Okou2, Okou1}).

When the moduli stack $X$ admits a perfect obstruction theory with a global presentation $E = [E^{-1} \to E^0]$, where $E^{-1}, E^0$ are locally free sheaves on $X$, one can define the virtual structure sheaf as
\begin{align*}
    [\sO_X\virt ]:= 0_{E_1}^! [\oO_{C_1}] \in K_0(X)
\end{align*}
where $E_1 = (E^{-1})^\vee$ and $C_1 = \cC_X \times_{\eE_X} E_1$.

However, there are many moduli stacks of interest which do not admit perfect obstruction theories, including moduli of simple complexes \cite{Inaba, Lieblich} and desingularizations of stacks of semistable sheaves and perfect complexes on Calabi-Yau threefolds \cite{KLS, Sav}.

In order to resolve this, in our previous paper \cite{KiemSavvas} we introduced a relaxed version of perfect obstruction theories, called almost perfect obstruction theories, which arise in the above moduli stacks. An almost perfect obstruction theory has an obstruction sheaf $\Ob_X$, which is the analogue of the sheaf $h^1(E^\vee)$, and induces an embedding of the coarse moduli sheaf $\fc_X$ of $\cC_X$ into $\Ob_X$ enabling us to define the virtual structure sheaf of $X$ as
\begin{align} \label{loc 1.1}
    [\sO_X\virt] := 0_{\Ob_X}^! [\oO_{\fc_X}] \in K_0(X).
\end{align}

Several techniques have been developed to handle virtual fundamental cycles and virtual structure sheaves arising from perfect obstruction theories on Deligne-Mumford stacks, such as the virtual torus localization of Graber-Pandharipande \cite{GrabPand, Qu}, the cosection localization of Kiem-Li \cite{KiemLiCosection, KiemLiKTheory}, virtual pullback \cite{Manolache, Qu} and wall crossing formulas \cite{KiemLi}. Often, combining these (cf. for example \cite{ChangKiemLi}) can be quite effective.

The aim of the present paper is to generalize the virtual torus localization and cosection localization formulas of virtual structure sheaves and their combination  to the setting of almost perfect obstruction theory.

\medskip

Roughly speaking, an almost perfect obstruction theory $\phi$ on a Deligne-Mumford stack $X$ consists of perfect obstruction theories
$$ \phi\lalp \colon E\lalp \lr \bL_{U\lalp} $$
on an \'{e}tale cover $\{ U\lalp \to X \}$ of $X$ satisfying appropriate combatibility conditions (cf. Definition~\ref{APOT}), which ensure that we have an obstruction sheaf $\fF = \Ob_X$ and an embedding of the coarse intrinsic normal cone $\fc_X$ into $\fF$, so that the virtual structure sheaf $[\oO_X\virt]$ can be defined as in \eqref{loc 1.1} above.

The definition of $[\oO_X\virt]$ is based crucially on the development of a $K$-theory group $K_0(\fF)$ of coherent sheaves on the sheaf stack $\fF$, so that $[\oO_{\fc_X}] \in \Coh(\fF)$, and the construction of the Gysin map $0_\fF^! \colon K_0(\fF) \to K_0(X)$ in \cite{KiemSavvas}.

In this paper, we investigate the $K$-theory of sheaf stacks in more detail and establish an appropriate version of descent theory for the category $\Coh(\fF)$ of coherent sheaves on $\fF$ and several functorial behaviors for $\Coh(\fF)$ and $K_0(\fF)$ and their properties. These include:
\begin{enumerate}
    \item Pullback functors for surjective homomorphisms $\gG \to \fF$.
    \item Pullback and pushforward functors for injective homomorphisms $\gG \to \fF$ with locally free quotient $L = \fF / \gG$.
    \item Pullback and pushforward functors for morphisms $\tau \colon Y \to X$ of the base.
\end{enumerate}

\medskip

With these well developed, after appropriate modifications, the standard arguments proving the virtual torus localization formula, cosection localization formula and their combination work in our setting.

Thus, when $X$ admits a $T=\bC^\ast$-action, and the cover $\lbrace U\lalp \to X \rbrace$ and obstruction theories $\phi\lalp$ are $T$-equivariant, we prove that (Theorem~\ref{46})
$$ [\oO_X\virt] = \iota_\ast \frac{[\oO_F\virt]}{e(N\virt)} \in K_0^T(X) \otimes_{\bQ[t,t^{-1}]} \bQ(t)$$
where $\iota \colon F \to X$ is the inclusion of the $T$-fixed locus, which admits an induced almost perfect obstruction theory, and we assume that the virtual normal bundle $N\virt$ of $F$ has a global resolution $[ N_0 \to N_1]$ by locally free sheaves.

Moreover, if there is a cosection $\sigma \colon \Ob_X \to \oO_X$ with vanishing locus $X(\sigma)$, we prove (Definition~\ref{42}, Proposition~\ref{43}) that the virtual structure sheaf $[\oO_X\virt]$ localizes canonically to an element $[\oO_{X,\loc}\virt] \in K_0(X(\sigma))$.

In the presence of a $T$-action such that the almost perfect obstruction theory is $T$-equivariant and a $T$-invariant cosection $\sigma$, we show that the virtual torus localization formula holds for the cosection localized virtual structure sheaves of $X$ and $F$ (Theorem~\ref{55}).

\medskip

An immediate application of the virtual torus localization formula is a $K$-theoretic wall crossing formula for simple $\bC^\ast$-wall crossings (Theorem~\ref{theorem 7.2}).

Another application is a $K$-theoretic refinement of the Jiang-Thomas theory of virtual signed Euler characteristics \cite{JiangThomas}. If $X$ is a Deligne-Mumford stack with a perfect obstruction theory with obstruction sheaf $\fF = \Ob_X$, then we show (Theorem~\ref{theorem 7.4}) that the dual obstruction cone $N = \Spec_X (\Sym \fF)$ admits a (symmetric) almost perfect obstruction theory with obstruction sheaf $\Ob_N = \Omega_N$. This is $T$-equivariant for the natural $T$-action with fixed locus $\iota \colon X \to N$. Additionally, there is a cosection $\sigma \colon \Ob_N \to \oO_N$. When $X$ is proper, we obtain (Theorem~\ref{thm 7.6}) the $K$-theoretic invariants
$$\chi \left( [\oO_{N,\loc}\virt] \right) \in \bQ, \quad \chi_t \left( \frac{[\oO_X\virt]}{e(E^\vee)}\right) \in \bQ(t)$$
refining the Jiang-Thomas virtual signed Euler characteristics.

\medskip \subsection*{Layout of the paper} \S\ref{background section} collects necessary background on the $K$-theory of sheaf stacks, almost perfect obstruction theories and virtual structure sheaves that we need from \cite{KiemSavvas}. In \S\ref{functoriality section} we study the descent theory and functoriality properties of coherent sheaves on sheaf stacks, which are then used throughout the rest of the paper. \S\ref{Scos} treats the cosection localization formula, \S\ref{Torus localization section} treats the virtual torus localization formula and \S\ref{combination section} their combination. In \S\ref{applications section} we apply these formulas to prove the $K$-theoretic wall crossing formula for simple $\bC^\ast$-wall crossings and construct a $K$-theoretic refinement of the Jiang-Thomas theory of virtual signed Euler characteristics. Finally, in the Appendix, we give the proof of the deformation invariance of the cosection localized virtual structure sheaf.

\medskip \subsection*{Acknowledgements} We would like to thank Dan Edidin for kindly answering our questions on localization in equivariant $K$-theory.

\medskip \subsection*{Notation and conventions} 
Everything in this paper is over the field $\bC$ of complex numbers. All stacks are of finite type and Deligne-Mumford stacks are separated. 

If $E$ is a locally free sheaf on a Deligne-Mumford stack $X$, we will use the term “vector bundle” to refer to its total space. If $\fF$ is a coherent sheaf on a Deligne-Mumford stack $X$, we will use the same letter to refer to the associated sheaf stack.

For a morphism $f:Y\to X$ of stacks and a coherent sheaf $\fF$ on $X$, its pullback $f^*\fF$ is sometimes denoted by $\fF|_Y$ when the map $f$ is clear from the context.
The bounded derived category of coherent sheaves on a stack $X$ is denoted by $D(X)$ and $\bL_{X/S}=L^{\ge -1}_{X/S}$ denotes the truncated cotangent complex for a morphism $X \to S$.
$T$ typically denotes the torus $\bC^*$.

\bigskip

\section{$K$-Theory on Sheaf Stacks, Almost Perfect Obstruction Theory and Virtual Structure Sheaf} \label{background section}

In \cite{KiemSavvas}, we introduced the notion of coherent sheaves on a sheaf stack $\fF$, defined the $K$-theory of $\fF$ and constructed a Gysin map $0^!_\cF$ of $\fF$, which enabled us to construct the virtual structure sheaf $[\sO_X\virt]\in K_0(X)$ for a \DM stack $X$ equipped with an almost perfect obstruction theory.
 
In this preliminary section, we collect necessary ingredients from \cite{KiemSavvas}.

\medskip \subsection{Sheaf stacks, local charts and common roofs} In what follows, $X$ will denote a Deligne-Mumford stack and $\fF$ a coherent sheaf on $X$.  

\begin{defi} \emph{(Sheaf stack)} The \emph{sheaf stack} associated to $\fF$ is the stack that to every morphism $\rho \colon W \to X$ from a scheme $W$ associates the set $\Gamma(W, \rho^\ast \fF)$. 
\end{defi}
 
By abuse of notation, we denote by $\fF$ the sheaf stack associated to a coherent sheaf $\fF$ on $X$.  
A sheaf stack is not algebraic in general and we need an appropriate notion of local charts for geometric constructions. 

\begin{defi} \emph{(Local chart)} \label{local chart}
A \emph{local chart} $Q=(U, \rho, E, r_E)$ for the sheaf stack $\fF$ consists of
\begin{enumerate}
\item an \'etale morphism $\rho:U\to X$ from a scheme $U$, and 
\item a surjective homomorphism $r_E:E\to \rho^* \fF=\fF|_U$ of coherent sheaves on $U$ from a locally free sheaf $E$ on $U$.
\end{enumerate}
We will call $U$ the \emph{base} of the chart $Q$. 
If $U$ is affine and $E$ is free, then the local chart $Q = (U, \rho, E, r_E)$ is called \emph{affine}.
\end{defi}

\begin{defi} \emph{(Morphism between local charts)}
Let $Q=(U, \rho, E, r_E)$ and $Q'=(U', \rho', E', r_{E'})$ be two local charts for $\fF$. A \emph{morphism} $\gamma \colon Q \to Q'$ is the pair $(\rho_\gamma, r_\gamma)$ of an \'{e}tale morphism $\rho_\gamma \colon U \to U'$ and a surjection $r_\gamma \colon E \to \rho_\gamma^* E'$ of locally free sheaves, such that the diagrams
\begin{align*}
\xymatrix{
U \ar[r]^-{\rho_\gamma} \ar[dr]_-{\rho} & U' \ar[d]^-{\rho'} \\
& X
} \ \textrm{} \ \xymatrix{
E \ar[r]^-{r_\gamma} \ar[dr]_-{r_E} & \rho_\gamma^* E' \ar[d]^-{\rho_\gamma^* r_{E'}} \\
& \fF|_U
}
\end{align*}
are commutative.

We say that $Q$ is a \emph{restriction} of $Q'$ and write $Q = Q'|_U$ if $E = \rho_\gamma^* E'$ and $r_\gamma$ is the identity morphism.
\end{defi}

The notion of a common roof enables us compare two local charts on $\fF$ with the same base $\rho \colon U \to X$.

\begin{defi} \emph{(Common roof)} \label{common roof defi}
Let $r:\cG\to \fF$ and $r':\cG'\to \fF$ be two surjective homomorphisms of coherent sheaves on a scheme $U$. Their fiber product is defined by
\beq\label{1}\cG \times_{\fF} \cG':=\mathrm{ker}\left( \cG\oplus \cG'\xrightarrow{(r,-r')} \fF\oplus \fF\xrightarrow{+} \fF \right)\eeq 
and we have a commutative diagram
\[\xymatrix{
\cG\times_{\fF}\cG'\ar[r]\ar[d] & \cG\ar[d] \\
\cG'\ar[r] & \cF
}\]
of surjective homomorphisms, which is 
universal among such diagrams of surjective homomorphisms in the obvious sense. 

Given two charts $Q=(U, \rho, E, r_E)$ and $Q'=(U, \rho, E', r_{E'})$ with the same quasi-projective base $U$, we can pick a surjective homomorphism $$W\lra E\times_{\fF|_U}E'$$ from a locally free sheaf $W$. Denoting the induced surjection $W\to  \fF|_U$ by $r_W$, we obtain a local chart $(U,\rho, W, r_W)$, which we call a \emph{common roof} of $Q$ and $Q'$.
\end{defi}

More generally, given two charts $Q=(U,\rho, E, r_E)$ and $Q'=(U',\rho',E',r_{E'})$ of the sheaf stack $\cF$, we let $V=U\times_XU'$ and have two local charts $Q|_V$ and $Q'|_V$ with the same base. A common roof of $Q|_V$ and $Q'|_V$ is called a \emph{common roof} of $Q$ and $Q'$.

\medskip \subsection{Coherent sheaves on a sheaf stack $\fF$} 

A \emph{coherent sheaf} $\aA$ on $\fF$ is an assignment to every local chart $Q = (U, \rho, E, r_E)$ of a coherent sheaf $\aA_Q$ on the scheme $E$ (in the \'{e}tale topology) such that for every morphism $\gamma \colon Q \to Q'$ between local charts there exists an isomorphism 
\beq\label{y2} r_\gamma^* \left( \rho_\gamma^*\aA_{Q'} \right) \lr \aA_Q\eeq
which satisfies the usual compatibilities for composition of morphisms. 
Note that we abusively write $\rho_\gamma^* \aA_{Q'}$ for the pullback of $\aA_{Q'}$ to $\rho_\gamma^* E'$ via the morphism of bundles $\rho_\gamma^* E' \to E'$ induced by $\rho_\gamma$.
A \emph{quasicoherent sheaf} on a sheaf stack is defined likewise. 

A homomorphism $f:\aA\to \bB$ of (quasi)coherent sheaves on $\fF$ is the data of a homomorphism
$f_Q:\aA_Q\to \bB_Q$ of (quasi)coherent sheaves on $E$ for each local chart $Q=(U,\rho,E,r_E)$ such that for every morphism $\gamma:Q\to Q'$ of local charts,
the diagram 
$$\xymatrix{
r_\gamma^* \left( \rho_\gamma^*\aA_{Q'} \right) \ar[r]\ar[d]_{f_{Q'}}& \aA_Q\ar[d]^{f_Q}\\
r_\gamma^* \left( \rho_\gamma^*\bB_{Q'} \right) \ar[r] & \bB_Q
}$$
is commutative where the horizontal arrows are \eqref{y2}. We say that a homomorphism $f:\aA\to \bB$ is an isomorphism if $f_Q$ is an isomorphism for each local chart $Q$. 

Exact sequences and the $K$-group $K_0(\fF)$ were defined in \cite{KiemSavvas} as follows.
\begin{defi} \emph{(Short exact sequence)} \label{ses}
Let $\aA, \bB, \cC$ be coherent sheaves on the sheaf stack $\fF$. A sequence 
\begin{align*}
    0 \lr \aA \lr \bB \lr \cC \lr 0
\end{align*}
of homomorphisms of coherent sheaves on $\fF$  
is \emph{exact} if for every local chart $Q=(U, \rho, E, r_E)$ on $\fF$ the sequence
\begin{align*}
    0 \lr \aA_Q \lr \bB_Q \lr \cC_Q \lr 0
\end{align*}
is an exact sequence of coherent sheaves on the scheme $E$.
\end{defi} 
Note that the morphism
$$E\mapright{r_\gamma} \rho_\gamma^*E'=U\times_{U'}E'\lra E'$$
is smooth and hence flat.  
We can likewise define the \emph{kernel} and \emph{cokernel} of a homomorphisms $f:\aA\to \bB$ of coherent sheaves on the sheaf stack $\fF$. Thus coherent and quasicoherent sheaves on $\fF$  form abelian categories $$\mathrm{Coh}(\fF) \subset \mathrm{QCoh}(\fF).$$ 

\begin{defi} 
The \emph{$K$-group} of coherent sheaves on $\fF$ is the group generated by the isomorphism classes $[\aA]$ of coherent sheaves $\aA$ on $\fF$, with relations generated by $[\bB] = [\aA] + [\cC]$ for every short exact sequence $$0 \lr \aA \lr \bB \lr \cC \lr 0.$$
\end{defi}
In other words, $K_0(\cF)$ is the Grothendieck group of the abelian category $\mathrm{Coh}(\cF)$. 

If $\fF=E$ is locally free, so that $\fF$ is an algebraic stack, then the above definitions recover the standard notions of short exact sequences and $K_0(\fF)$ for the vector bundle $E$.

\medskip \subsection{$K$-theoretic Gysin and pullback maps} \label{Koszul homology section} 
Let $Q = (U, \rho, E, r_E)$ be a local chart for the sheaf stack $\fF$ and denote the vector bundle projection map $E \to U$ by $\pi_E$. The tautological section of the pullback $\pi_E^* E$ induces an associated Koszul complex $$\cK(E):=\wedge^\bullet \pi_E^* E^\vee$$ that resolves the structure sheaf $\oO_U$ of the zero section of $\pi_E$.

\begin{defi}
For any local chart $Q = (U, \rho, E, r_E)$ and a coherent sheaf $\aA$ on $\fF$, the \emph{$i$-th Koszul homology sheaf} $$\hH_Q^i(\aA)\in \mathrm{Coh}(U)$$ of $\aA$ with respect to $Q$ is defined as the homology of the complex 
$$\cK(E)\otimes_{\sO_E} \aA_Q=\wedge^\bullet \pi_E^* E^\vee \otimes_{\oO_E} \aA_Q$$ 
in degree $-i$.
\end{defi}

Using common roofs and the standard descent theory for coherent sheaves on algebraic stacks, the following is proven in \cite{KiemSavvas}.

\begin{thm-defi}
Let $\aA$ be a coherent sheaf on a sheaf stack $\fF$ over a Deligne-Mumford stack $X$. The coherent sheaves $\hH_Q^i(\aA)\in \mathrm{Coh}(U)$ glue canonically to a coherent sheaf $\hH_\cK^i(\aA)\in \mathrm{Coh}(X)$ on $X$, which is defined to be the $i$-th Koszul homology sheaf of $\aA$.
\end{thm-defi} 

The Koszul homology sheaves of $\aA$ were then used to define a $K$-theoretic Gysin map $0_\fF^! \colon K_0(\fF) \to K_0(X)$.

\begin{defi}\label{32} \emph{($K$-theoretic Gysin map)}
The $K$-theoretic Gysin map  is defined by the formula
\begin{align}\label{31}
 0_\fF^! \colon K_0(\fF) \lra K_0(X),\quad   0_{\fF}^! [\aA] = \sum_{i \geq 0} (-1)^i [ \hH_\cK^i(\aA) ] \in K_0(X)
\end{align}
where $\aA$ is a coherent sheaf on $\fF$.
\end{defi}

\medskip \subsection{Almost perfect obstruction theory and virtual structure sheaf}\label{S4}

A perfect obstruction theory on a morphism $X \to S$, where $X$ is a scheme and $S$ a smooth Artin stack, is a morphism
$$\phi \colon E \lra \bL_{X/S}$$
in $D(X)$, where $E$ is a perfect complex of amplitude $[-1,0]$, satisfying that $h^{-1}(\phi)$ is surjective and $h^0(\phi)$ is an isomorphism.

We refer to the sheaf $\Ob_X := h^1 (E^\vee)$ as the obstruction sheaf associated to the perfect obstruction theory $\phi$. Then $\phi$ induces a Cartesian diagram
\begin{align} \label{loc 3.1}
    \xymatrix{
    \cC_{X/S} \ar[r] \ar[d] & \nN_{X/S} := h^1 /h^0 (\bL_{X/S}^\vee) \ar[rr]^-{h^1 /h^0 (\phi^\vee)} \ar[d] && \eE := h^1 / h^0 (E^\vee) \ar[d] \\
    \fc_{X/S} \ar[r] & \fn_{X/S} = h^1(\bL_{X/S}^\vee) \ar[rr]_-{h^1(\phi^\vee)} && \Ob_X= h^1 (E^\vee)
    }
\end{align}
where $\cC_{X/S}$ and $\nN_{X/S}$ are the intrinsic normal cone and intrinsic normal sheaf of $X$ over $S$ respectively, while $\fc_{X/S}$ and $\fn_{X/S}$ are their coarse moduli sheaves. All the horizontal arrows are closed embeddings.

Some interesting moduli spaces in algebraic geometry such as the moduli space of derived category objects do not admit a perfect obstruction theory. 
In \cite{KiemSavvas}, we introduced a weaker notion with which moduli stacks for generalized Donaldson-Thomas invariants are equipped. 
\begin{defi} \emph{(Almost perfect obstruction theory)} \label{APOT}
Let $X \to S$ be a morphism, where $X$ is a Deligne-Mumford stack of finite presentation and $S$ is a smooth Artin stack of pure dimension. An \emph{almost perfect obstruction theory} $\phi$ consists of an \'{e}tale covering $\lbrace X_\alpha \to X \rbrace_{\alpha \in A}$ 
of $X$ and perfect obstruction theories $\phi_\alpha \colon E_\alpha \to \bL_{X_\alpha / S}$ of $X_\alpha$ over $S$ such that the following hold. 
\begin{enumerate}
\item For each pair of indices $\alpha, \beta$, there exists an isomorphism \begin{align*}
\psi_{\alpha \beta} \colon \Ob_{X_\alpha} \vert_{X_{\alpha\beta}} \lra \Ob_{X_\beta} \vert_{X_{\alpha\beta}}
\end{align*}
so that the collection $\lbrace \Ob_{X\lalp}=h^1(E_\alpha^\vee), \psi\lab \rbrace$ gives descent data of a sheaf $\Ob_X$, called the obstruction sheaf, on $X$.
\item For each pair of indices $\alpha, \beta$, there exists an \'{e}tale covering $\lbrace V_\lambda \to X\lab \rbrace_{\lambda \in \Gamma}$ of $X\lab=X_\alpha\times_XX_\beta$ such that for any $\lambda$, the perfect obstruction theories $E_\alpha \vert_{V_{\lambda}}$ and $E_\beta \vert_{V_{\lambda}}$ are isomorphic and compatible with $\psi\lab$. This means that there exists an isomorphism
\begin{align*}
    \eta_{\alpha\beta\lambda} \colon E_\alpha \vert_{V_{\lambda}} \lr E_\beta \vert_{V_{\lambda}}
\end{align*} 
in $D(V_\lambda)$ fitting in a commutative diagram
\begin{align} \label{compatibility data APOT}
    \xymatrix{
    E_\alpha \vert_{V_{\lambda}} \ar[d]_-{\phi\lalp|_{V_\lambda}} \ar[r]^-{ \eta_{\alpha\beta\lambda}} & E_\beta \vert_{V_{\lambda}} \ar[d]^-{\phi\lbet|_{V_\lambda}} \\
    \bL_{X_\alpha / S}|_{V_\lambda} \ar[r] \ar[dr]_\cong & \bL_{X_\beta / S}|_{V_\lambda} \ar[d]^\cong \\
    & \bL_{V_\lambda / S}
    }
\end{align}
which moreover satisfies $h^1(\eta_{\alpha\beta\lambda}^\vee) = \psi\lab^{-1}|_{V_\lambda}$.
\end{enumerate}
\end{defi}

Suppose that the morphism $X \to S$ admits an almost perfect obstruction theory. Then the definition implies that the closed embeddings given in diagram \eqref{loc 3.1}
$$h^1(\phi\lalp^\vee) \colon \fn_{U\lalp/S} \lr \Ob_{X\lalp}$$
glue to a global closed embedding
$$j_\phi \colon \fn_{X/S} \lr \Ob_{X}$$
of sheaf stacks over $X$.
Therefore, the coarse intrinsic normal cone stack $\fc_{X/S}$ embeds as a closed substack into the sheaf stack $\Ob_X$.

\begin{defi} \cite{KiemSavvas} \emph{(Virtual structure sheaf)} \label{virtual structure sheaf}
Let $X \to S$ be as above, together with an almost perfect obstruction theory $\phi \colon X \to S$. The \emph{virtual structure sheaf} of $X$ associated to $\phi$ is defined as
\begin{align*}
    [\oO_X\virt] := 0_{\Ob_X}^! [\oO_{\fc_{X/S}}] \in K_0(X).
\end{align*}
\end{defi}

\medskip

It is straightforward to generalize Definition \ref{APOT} to the relative setting of a morphism $X\to Y$ of \DM stacks, by considering an \'etale cover $\{Y_\alpha\to Y\}$ and relative perfect obstruction theories 
on the morphisms $X_\alpha=X\times_Y Y_\alpha\to Y_\alpha$. We leave the detail to the reader. 

\medskip

In the subsequent sections, we will generalize the torus localization theorem \cite{GrabPand, Qu}, the cosection localization theorem \cite{KiemLiCosection, KiemLiKTheory} and the wall crossing formula \cite{ChangKiemLi} to the virtual structure sheaves associated to \emph{almost perfect obstruction theories}.

\bigskip

\section{Functorial Behavior of Coherent Sheaves on Sheaf Stacks} \label{functoriality section}

In this section, we investigate the functorial behavior of coherent sheaves on the sheaf stack $\cF$. 
These functoriality properties will be our fundamental tools in generalizing the localization theorems to almost perfect obstruction theories in the subsequent sections. 

We will first establish a descent theory for coherent sheaves on $\cF$. Using this, we will prove that there are pullbacks of a coherent sheaf $\aA$ on $\cF$ by a morphism $\tau:Y\to X$ of the base and by a surjective homomorphism $f:\cG\to \cF$ of the sheaves as well as by an injective homomorphism $f:\cG\to \cF$ with locally free quotient $L=\cF/\cG$.  

Moreover, we will see that there are (higher) pushforwards of a coherent sheaf $\aA$ on $\tau^*\cF$ to coherent sheaves $R^i{\tau}_*\aA$ on $\cF$ for any proper morphism $\tau:Y\to X$. When $\cG\to \cF$ is an injective homomorphism of coherent sheaves on $X$ with locally free quotient, we will define the pushforward of a coherent sheaf on $\cG$ to $\cF$. These pushforwards and pullbacks satisfy the expected adjunction properties. 

\medskip \subsection{Coherent descent theory}

Let $f:\cG\to \cF$ be a surjective homomorphism of coherent sheaves on a \DM stack $X$. Then any local chart $$Q=(\rho:U\to X, r_E:E\twoheadrightarrow \cG)$$ of  $\cG$ induces the local chart
$$f_*Q=(\rho:U\to X, f_*r_E:E\twoheadrightarrow \cG\twoheadrightarrow \cF)$$
of $\cF$. Thus a coherent sheaf $\aA\in \mathrm{Coh}(\cF)$ gives a coherent sheaf $\aA_{f_*Q}\in \mathrm{Coh}(E)$ for every local chart $Q$ of $\cG$. It is easy to see that this assignment is a coherent sheaf on $\cG$, denoted by $f^*\aA$ or $\aA|_\cG$. In this way, we obtain a functor
\beq\label{9} f^*:\mathrm{Coh}(\cF)\lra \mathrm{Coh}(\cG)\eeq
and a homomophism 
\beq \label{10} f^*:K_0(\cF)\lra K_0(\cG)\eeq
since $f^*$ preserves exact sequences. From the construction, it is straightforward that if 
$g:\cG'\to \cG$ is another surjective homomorphism, then we have the equality 
\beq\label{6} (f\circ g)^*=g^*\circ f^*.\eeq

Using the above pullback functor, we can develop a descent formalism for sheaves on sheaf stacks. 
To begin with, it is convenient to extend the notion of local charts. 
\begin{defi} \emph{(Coherent chart)} Let $\cF$ be a coherent sheaf on a \DM stack $X$. 
A \emph{coherent chart} on the sheaf stack $\fF$ is the datum of a quadruple 
$$\pP = (U, \rho, \cG, r_\cG)$$
where $U$ is a scheme, $\rho \colon U \to X$ is an \'{e}tale morphism and $r_\cG \colon \cG \to \rho^* \fF$ is a surjective homomorphism of coherent sheaves on $U$.
\end{defi}
Of course, 
when $\cG$ is locally free, a coherent chart is the same as a local chart.

Morphisms between coherent charts are defined in the same way as morphisms between local charts. The main advantage of coherent charts is that one has a natural fiber product. 

\begin{defi} \emph{(Fiber product of coherent charts)}
Let $\pP = (U, \rho, \cG, r_\cG)$ and $\pP'=(U', \rho', \cG', r_{\cG'})$ be two coherent charts for $\fF$. Let $V = U \times_X U'$ and $\rho_V \colon V \to X$ the natural map to $X$. 
Using \eqref{1}, we have the fiber product 
$$\cG \times_{\fF} \cG':=\cG|_V\times_{\cF|_V}\cG'|_V.$$ 
Denote the induced surjection $\cG \times_{\fF} \cG' \to \fF|_V$ by $r_\cG \times_\fF r_{\cG'}$. The coherent chart $(V, \rho_V, \cG \times_{\fF} G', r_\cG \times_\fF r_{\cG'})$ is called the \emph{fiber product} of the coherent charts $\pP$ and $\pP'$ and denoted by $\pP \times_\fF \pP'$. There are natural morphisms of coherent charts $\pP \times_\fF \pP' \to \pP$ and $\pP \times_\fF \pP' \to \pP'$.
\end{defi}

\begin{rmk}
If $\pP$ and $\pP'$ are local charts, then it is easy to see that a common roof of $\pP$ and $\pP'$ by Definition~\ref{common roof defi} is a local chart for the fiber product $\pP \times_{\cF} \pP'$.
\end{rmk}

\begin{defi} \emph{(Atlas on $\fF$)}
An \emph{atlas} on a sheaf stack $\fF$ is a collection $$\lbrace \pP\lalp = (U\lalp, \rho\lalp, \cG\lalp, r\lalp) \rbrace_{\alpha \in \Gamma}$$ of coherent charts on $\fF$ such that the morphisms $\lbrace \rho\lalp \colon U\lalp \to X \rbrace$ give an \'{e}tale cover of $X$. We write $U\lab = U\lalp \times_X U\lbet$ and $\pP\lab = (U\lab, \rho\lab, \cG\lab, r\lab)$ for the fiber product $\pP\lalp \times_\fF \pP\lbet$. We use similar notation $U\labc$ and $\pP\labc$ for triples of indices.
\end{defi}

If $\aA$ is a coherent sheaf on $\cF$ and $\{\pP_\alpha\}$ is an atlas on $\cF$, then we have 
the pullbacks $$\aA_\alpha\in \mathrm{Coh}(\cG_\alpha)$$ of $\aA$ to $\cG_\alpha$ by the surjections $r_\alpha:\cG_\alpha\to \cF|_{U_\alpha}$. Moreover, the pullbacks $\aA_\alpha$ and $\aA_\beta$ to $\cG\lab$ are isomorphic to $\aA|_{\cG\lab}$. By \eqref{6}, we have the isomorphisms 
\beq\label{60} g\lab:\aA\lalp|_{\cG\lab}\mapright{=} \aA|_{\cG\lab}\mapleft{=} \aA\lbet|_{\cG\lab},\eeq
which satisfy the equality 
\beq\label{5} g_{\gamma \alpha}|_{\cG\labc} \circ g_{\beta\gamma}|_{\cG\labc} \circ g\lab |_{\cG\labc} = \mathrm{id}.\eeq

\begin{defi} \emph{(Descent datum)}
A \emph{descent datum} on a sheaf stack $\fF$ consists of 
\begin{enumerate}
    \item[(a)] an atlas $\lbrace \pP\lalp = (U\lalp, \rho\lalp, \cG\lalp, r\lalp) \rbrace_{\alpha \in A}$ on $\fF$;
    \item[(b)] a coherent sheaf $\aA\lalp$ on the sheaf stack $\cG\lalp$ for every index $\alpha$;
    \item[(c)] an isomorphism
    $$g\lab \colon \aA\lalp |_{\cG\lab} \lr \aA\lbet |_{\cG\lab}$$
    of coherent sheaves on the sheaf stack $\cG\lab$ for every pair $\alpha, \beta$, 
\end{enumerate}
such that the cocycle condition \eqref{5} holds for every triple $\alpha, \beta, \gamma$ of indices.

We say that a descent datum is \emph{effective} if there exists a coherent sheaf $\aA$ on $\fF$ such that we have $\aA\lalp = \aA |_{\cG\lalp}$ for all $\alpha$ and the $g\lab$ are the naturally induced isomorphisms \eqref{60} on fiber products. We then say that the datum \emph{descends} or \emph{glues} to the sheaf $\aA$.
\end{defi}

For example, let $f:\cG\to \cF$ be a surjective homomorphism of coherent sheaves on a \DM stack $X$. The single chart $(\id:X\to X, f:\cG\to \cF)$ is then an atlas for $\cF$. The fiber product
$\cG\times_\cF\cG$ fits into the commutative diagram
\beq\label{7} \xymatrix{
\cG\times_\cF\cG\ar[r]^-p \ar[d]_q & \cG\ar[d]^f\\
\cG\ar[r]_f & \cF}.\eeq
A descent datum for this atlas consists of a coherent sheaf $\bB$ on $\cG$ and an isomorphism 
$g:p^*\bB\to q^*\bB$. For a local chart $(\rho:U\to X, r_E:E\to \rho^*\cF)$ with affine base $U$, we pick a surjection 
$W\to E\times_{\cF|_U}\cG|_U$ from a locally free sheaf $W$ on $U$ so that we have surjections
$$r_W:W\lra E\times_{\cF|_U}\cG|_U\lra \cG|_U, \quad W\times_EW\lra \cG\times_\cF\cG|_U$$
and a fiber diagram
\beq\label{8}\xymatrix{
W\times_EW\ar[r]^-{p_W} \ar[d]_{q_W} & W\ar[d]\\
W\ar[r] & E
}\eeq
lying over \eqref{7}. 
As $\bB\in \mathrm{Coh}(\cG)$, we have a coherent sheaf $\bB_W$ on $W$ and an isomorphism
$p_W^*\bB_W\cong q_W^*\bB_W$. Since $W\to E$ is a surjective homomorphism of vector bundles, the coherent sheaf $\bB_W$ descends to a coherent sheaf $\bB_E$ on $E$. It is easy to see that $\bB_E$ is independent of the choice of $W$ and we have a coherent sheaf $\bB_E$ for any local chart of $\cF$ by considering an affine cover of the base scheme $U$. In this way, we obtain a coherent sheaf on $\cF$. 

By the same argument, we deduce the following. 
\begin{prop} \label{descent on F} \emph{(Descent for sheaves on $\fF$)} Any descent datum on a sheaf stack $\fF$ is effective and descends to a coherent sheaf on $\fF$, which is uniquely determined (up to canonical isomorphism).
\end{prop}

\begin{proof}
This is a standard descent argument using smooth descent for sheaves on schemes repeating the reasoning we saw above on an atlas for $\fF$ and keeping track of all the indices involved. We leave the details to the reader. 
\end{proof}

\medskip \subsection{Pullbacks of coherent sheaves on sheaf stacks} 
In this subsection, we will see more ways to pull back coherent sheaves or $K$-theory classes, other than \eqref{9} and \eqref{10}. 

\medskip

Let $\cF$ be a coherent sheaf on a \DM stack $X$. Then all local charts of the sheaf stack $\cF$ form an atlas in the obvious way. Suppose we have a short exact sequence
$$0\lra \cG\mapright{f} \cF\lra L\lra 0$$
of coherent sheaves on $X$ with $L$ locally free. Then any local chart $Q=(U,\rho,E, r_E)$ of $\cF$ gives rise to the commutative diagram of exact sequences
\beq\label{16}\xymatrix{
0\ar[r] & E'\ar[r]^{\tilde{f}} \ar[d]_{r_{E'}}& E\ar[r]\ar[d]^{r_E} & L|_U\ar[r]\ar@{=}[d] & 0\\
0\ar[r] & \cG|_U\ar[r]_f & \cF|_U\ar[r] & L|_U\ar[r] & 0
}\eeq
whose vertical arrows are all surjective. Note that $E'$ is locally free and 
it is straightforward to see that the local charts 
$Q'=(U,\rho, E',r_{E'})$ form an atlas on the sheaf stack $\cG$. 

Given a coherent sheaf $\aA_Q$ on $E$, we can pull  it back to $E'$ to get a coherent sheaf $\aA'_{Q'}=\tilde{f}^*\aA_Q$.
If $\aA\in \mathrm{Coh}(\cF)$, we then obtain a descent datum $\{\aA'_{Q'}\in \mathrm{Coh}(E')\}$ with the atlas $\{Q'\}$ on $\cG$. By Proposition \ref{descent on F}, we thus obtain a sheaf $\aA'$ on $\cG$, which we denote by $f^*\aA\in \mathrm{Coh}(\cG)$. Likewise, we can pull back homomorphisms in $\Coh(\cF)$ to $\Coh(\cG)$. We thus obtain the pullback functor 
\beq\label{11} f^*:\Coh(\cF)\lra \Coh(\cG), \eeq
for an injective homomorphism $f:\cG\to \cF$ of coherent sheaves on $X$ whose cokernel $L=\cF/\cG$ is locally free.

Unfortunately, this functor $f^*$ is not exact and hence does not induce a homomorphism on $K$-groups. 
To get an exact functor, we should use the left derived pullback: Letting $\pi_E:E\to U$ denote the bundle projection, we have a tautological section $s$ of $\pi^*_EL$ whose vanishing locus is $E'$. 

We then have the Koszul complex $\cK(L)$ and the homology sheaves of the complex $\cK(L) \otimes \aA_Q$ for each local chart $Q$ of $\cF$ give us descent data with respect to the atlas $\{ Q' \}$ on $\cG$. By Proposition~\ref{descent on F}, we get coherent sheaves 
$$L_0 f^* \aA = f^* \aA, \quad L_1f^* \aA,\quad L_2 f^* \aA,\quad \cdots$$ on $\cG$
(with degrees up to the rank of $L$), where each $L_i f^* \aA$ is the descent of the homology sheaves in degree $-i$ of the complexes $\cK(L) \otimes \aA_Q$ for local charts $Q$.
We thus obtain the pullback homomorphism 
\beq\label{12} f^*:K_0(\cF)\lra K_0(\cG),\quad [\aA]\mapsto \sum_{i \geq 0} (-1)^i [L_if^*\aA].\eeq
It is straightforward to check that \eqref{12} is well defined. 

The pullback \eqref{12} is compatible with the Gysin map $0^!_\cF$ as follows. 
\begin{lem}\label{14} 
Let $f:\cG\to \cF$ be an injective homomorphism of coherent sheaves on $X$ whose cokernel $L=\cF/\cG$ is locally free. Then we have the identity
\[  0^!_\cG\circ f^*=0^!_\cF.\]
\end{lem}

\begin{proof}
For any affine local chart $Q=(U, \rho, E, r_E)$, we can split $E$ as the direct sum $E'\oplus L$ and we have an induced isomorphism $\cK(E)\cong \cK(E')\otimes \cK(L)$ of Koszul complexes. Given a coherent sheaf $\aA$ on $\fF$, we then obtain an isomorphism
$$\cK(E) \otimes \aA_Q \cong \cK(E') \otimes (\cK(L) \otimes \aA_Q)$$
The spectral sequence for the double complex on the right hand side gives an equality in $K$-theory for any $\ell \geq 0$
$$[\hH_Q^{\ell}(\aA)] = \sum_{i+j = \ell} (-1)^j [\hH_{Q'}^i (L_j f^\ast \aA)] \in K_0(U).$$
Since the spectral sequence maps are functorial with respect to morphisms between local charts and independent of the choice of splitting of $E$, the above equality is true globally for Koszul homology sheaves using descent and the lemma follows.
\end{proof}

\medskip

We can also pull back coherent sheaves on a sheaf stack by a morphism of the underlying \DM stack. 
Let $\tau:Y\to X$ be a morphism of \DM stacks and $\cF$ be a coherent sheaf on $X$. As we saw above, local charts $Q=(U,\rho, E, r_E)$ of the sheaf stack $\cF$ form an atlas on $\cF$. 
The pullbacks 
\beq\label{20} Q'=(\rho':U'=Y\times_XU\to Y, r_{E'}=\tau^*r_E:E'=\tau^*E\to \tau^*\cF)\eeq 
form an atlas of the sheaf stack $\tau^*\cF$ on $Y$. By the induced morphism $\tilde{\tau}:E'=\tau^*E\to E$, we can pull back coherent sheaves on $E$ to $E'$. 
If $\aA\in \Coh(\cF)$ is a coherent sheaf on $\cF$, the coherent sheaves $\{\tilde{\tau}^*\aA_Q\in \Coh(E')\}$ form a descent datum and hence glue to a coherent sheaf $\tau^*\aA$ on $\tau^*\cF$. The same holds for homomorphisms in $\Coh(\cF)$. We thus obtain the pullback functor
\beq\label{15} \tau^*:\Coh(\cF)\lra \Coh(\tau^*\cF)\eeq
for a morphism $\tau:Y\to X$ of \DM stacks and a sheaf stack $\cF$ over $X$. 

If $\tau$ is flat, then $\tilde{\tau}:E'\to E$ is flat and $\tilde{\tau}^*$ is exact. Hence we obtain a homomorphism
\beq\label{13}\tau^*:K_0(\cF)\to K_0(\tau^*\cF)\eeq
when $\tau:Y\to X$ is flat. 

\medskip

The pullback $\tau^*$ by a flat morphism is compatible with the Gysin maps. 
\begin{lem}\label{4}
Let $\tau:Y\to X$ be a flat morphism of \DM stacks and $\cF$ be a coherent sheaf on $X$. 
Then the pullback \eqref{13} fits into the commutative diagram
\beq\label{3}
\xymatrix{ K_0(\cF)\ar[r]^{\tau^*} \ar[d]_{0^!_\cF} & K_0(\tau^*\cF)\ar[d]^{0^!_{\tau^*\cF}}\\
K_0(X)\ar[r]_{\tau^*} & K_0(Y).}
\eeq
\end{lem}
\begin{proof} The lemma follows from the isomorphism $\tau^*\cK(E)\cong \cK(\tau^*E)$ of Koszul complexes.
\end{proof}

\medskip \subsection{Pushforwards of coherent sheaves on sheaf stacks}
In this subsection, we will consider pushforwards on sheaf stacks. 

Suppose we have a short exact sequence
$$0\lra \cG\mapright{f} \cF\lra L\lra 0$$
of coherent sheaves on a \DM stack $X$ with $L$ locally free. For any local chart
$Q=(U,\rho,E,r_E)$, we have a commutative diagram \eqref{16} of exact sequences
and $Q'=(U,\rho,E', r_{E'})$ is a local chart for $\cG$.  
Given a coherent sheaf $\bB\in \Coh(\cG)$, we have a coherent sheaf $\bB_{Q'}$ on $E'$ and its pushforward  
$\tilde{f}_*\bB_{Q'}\in \Coh(E)$ by the closed immersion $\tilde{f}$.
By the base change property, the assignment $Q\mapsto \tilde{f}_*\bB_{Q'}$ is a coherent sheaf on $\cF$, which we denote by $f_*\bB\in \Coh(\cF)$. Likewise, we can pushforward a homomorphism in $\Coh(\cG)$ to $\Coh(\cF)$ by $f$. We thus obtain a functor
\beq\label{17} f_*:\Coh(\cG)\lra \Coh(\cF),\eeq
which is exact as $\tilde{f}$ is a closed immersion. So we have a homomorphism
\beq\label{18} f_*: K_0(\cG)\lra K_0(\cF)\eeq
for an injective homomorphism $f:\cG\to \cF$ of coherent sheaves on $X$ with locally free cokernel $L=\cF/\cG$. 

By construction, the natural adjunction homomorphisms 
$$\aA_Q\lra \tilde{f}_*\tilde{f}^*\aA_Q, \quad \tilde{f}^*\tilde{f}_*\bB_{Q'}\lra \bB_{Q'}$$
induce the natural homomorphisms
\beq\label{19} \aA\lra f_*f^*\aA,\quad f^*f_*\bB\lra \bB\eeq
in $\Coh(\cF)$ and $\Coh(\cG)$ respectively. 

\medskip

Let $\tau:Y\to X$ be a proper morphism of \DM stacks. Let $\cF$ be a coherent sheaf on $X$.
A local chart $Q=(U,\rho,E,r_E)$ of $\cF$ induces a local chart $Q'$, defined  by \eqref{20}, of the sheaf stack  $\tau^*\cF$
and a proper morphism $$\tilde{\tau}:E'=\tau^*E\to E.$$ 

Let  $\bB\in \Coh(\tau^*\cF)$ be a coherent sheaf on $\tau^*\cF$. 
Then we have the coherent sheaves
$$\tilde{\tau}_*\bB_{Q'}, \quad R^i\tilde{\tau}_*\bB_{Q'}\in \Coh(E)$$
by taking the direct and higher direct images. 
By the base chage property \cite[III, Proposition 9.3]{Hartshorne}, the assignments 
$$Q\mapsto \tilde{\tau}_*\bB_{Q'}, \quad Q\mapsto R^i\tilde{\tau}_*\bB_{Q'}$$
are coherent sheaves on $\cF$, which we denote by $\tau_*\bB=R^0\tau_*\bB$ and $R^i\tau_*\bB$ respectively. 
We thus obtain functors
\beq\label{21} \tau_*=R^0\tau_*, R^i\tau_*:\Coh(\tau^*\cF)\lra \Coh(\cF)\eeq
for $i>0$, which give us the homomorphism
\beq\label{22} \tau_*:K_0(\tau^*\cF)\lra K_0(\cF), \quad [\bB]\mapsto \sum_i(-1)^i[R^i\tau_*\bB]\eeq
for a proper morphism $\tau:Y\to X$. 

\begin{lem}\label{35}
For a proper morphism, the diagram 
\[\xymatrix{
K_0(\tau^*\cF)\ar[d]_{0^!_{\tau^*\cF}}\ar[r]^{\tau_*} & K_0(\cF)\ar[d]^{0^!_\cF}\\
K_0(Y)\ar[r]^{\tau_*} & K_0(X)
}\]
commutes.
\end{lem}
\begin{proof}
The lemma follows from Definition \ref{32} and the projection formula \cite[III, Exer. 8.3]{Hartshorne}.  
\end{proof}

As above, the usual adjunction homomorphisms give us homomorphisms
\beq\label{23} \aA\lra \tau_*\tau^*\aA, \quad \tau^*\tau_*\bB\to \bB\eeq
in $\Coh(\cF)$ and $\Coh(\tau^*\cF)$ respectively.

\section{Cosection Localization}\label{Scos}

Let $X \to S$ be a morphism, where $X$ is a Deligne-Mumford stack of finite type and $S$ a smooth quasi-projective curve, together with a perfect obstruction theory $\phi$ and a cosection $\sigma \colon \Ob_\phi \to \oO_X$. 
In \cite{KiemLiCosection} and \cite{KiemLiKTheory}, it is shown that the virtual fundamental cycle $[X]\virt \in A_* (X)$ and virtual structure sheaf $[\oO_X\virt] \in K_0(X)$ localize to the vanishing locus $X(\sigma)$ of $\sigma$, being the pushforward of localized classes in $A_*(X(\sigma))$ and $K_0(X(\sigma))$ respectively in a canonical way. 

In this section, we prove the analogous cosection localization for virtual structure sheaves induced by an \emph{almost} perfect obstruction theory.

\medskip \subsection{Cosection localization for perfect obstruction theory}\label{S4.1}
Let $X\to S$ be as above. Let $\rho:U \to X$ be an \'etale morphism from a scheme $U$ and let $\phi \colon E \to \bL_{U / S}^{\geq -1}$ be a perfect obstruction theory with $E = [E^{-1} \to E^0]$ a global resolution by vector bundles. Moreover, suppose that we have a morphism
\begin{align*}
    \sigma \colon \Ob_U \lr \oO_U,
\end{align*}
which, following \cite{KiemLiCosection}, we refer to as a \emph{cosection}.

Let $\underline{\sigma}$ be the composition $E_1 = (E^{-1})^\vee \to \Ob_U \to \oO_U$, $U(\sigma)$ the vanishing locus of $\sigma$, $U^\circ = U - U(\sigma)$ and 
\begin{align} \label{E(sigma)}
    E_1({\sigma}) = E_1|_{U(\sigma)} \cup \ker \left( \underline{\sigma}|_{U^\circ} \colon E_1 |_{U^\circ} \lr \oO_{U^\circ} \right).
\end{align}

In \cite{KiemLiKTheory}, the authors define a localized Gysin map
$$0^!_{E_1, {\sigma}} \colon K_0(E_1(\sigma)) \to K_0(U(\sigma)).$$

We recall the construction. Let $\tau \colon \tilde{U} \to U$ be the blowup of $U$ along $U(\sigma)$ with exceptional divisor $D$. If $\tilde{E}_1 = \tau^* E_1$, we have an induced surjection
\begin{align*}
   \tilde{\underline{\sigma}} \colon \tilde{E}_1 \lr \oO_{\tU}(-D).
\end{align*}
Let $E_1' = \ker \ti\sun$ and 
\begin{align} \label{E_1'}
    \ti{\tau} \colon E_1' \lra E_1(\sigma)
\end{align}
be the morphism induced from $\tE_1 \to E_1$. 

For any coherent sheaf $\aA$ on $E_1(\sigma)$, we have a natural morphism by adjunction 
\begin{align*}
    \eta_\aA \colon \aA \lr \ti{\tau}_* \ti{\tau}^* \aA.
\end{align*} 
Since $\eta_\aA$ is an isomorphism over $U^\circ$, the sheaves $\ker(\eta_\aA), \coker(\eta_\aA)$ and $R^i \ti{\tau}_* \ti{\tau}^* \aA$ for $i \geq 1$ are supported on $E_1|_{U(\sigma)}$.

We may therefore define
\begin{align} \label{expression for R_A}
    R_\aA := [\ker(\eta_\aA)] -[\coker(\eta_\aA)] - \sum_{i \geq 1}(-1)^i [R^i \ti{\tau}_* \ti{\tau}^* \aA] \in K_0(E_1|_{U(\sigma)}).
\end{align}

\begin{defi} \emph{\cite{KiemLiKTheory}} The cosection localized Gysin map is given by the formula
\begin{align} \label{localized Gysin map}
    0^!_{E_1, {\sigma}}[\aA] := (\tau|_D)_* \left( D^\vee \cdot 0_{E_1'}^! [\ti{\tau}^* \aA] \right) + 0^!_{E_1|_{U(\sigma)}}R_\aA \in K_0(U(\sigma)),
\end{align}
where $D^\vee \cdot [\aA'] = [\oO_{\ti{X}} \to \oO_{\ti{X}}(D)] \otimes [\aA']$.
\end{defi}

Let $C_1 := \fc_{X/S} \times_{\Ob_X} E_1 \sub E_1$ be the obstruction cone of the perfect obstruction theory $\phi$. In \cite{KiemLiCosection}, it is shown that $C_1$ has reduced support in $E_1(\sigma)$. Therefore, if we let $I$ denote the ideal sheaf of $(C_1)^\rred \sub C_1$, the sheaves $\aA_j = I^j \oO_{C_1} / I^{j+1} \oO_{C_1}$ for $j \geq 0$ are naturally coherent sheaves on $E_1(\sigma)$ and moreover
\begin{align*}
    [\oO_{C_1}] = \sum_{j \geq 0} [\aA_j] \in K_0(E_1(\sigma))
\end{align*} 
where the summation is finite, since $I^j = 0$ for large enough $j$.

\begin{defi-thm}\label{24} \emph{\cite{KiemLiKTheory}} The cosection localized virtual structure sheaf on $U$ is defined by
\begin{align*}
[\oO_{U,\loc}\virt] = 0^!_{E_1, {\sigma}}[\oO_{C_1}] := \sum_{j \geq 0} 0^!_{E_1, {\sigma}} [\aA_j] \in K_0(U(\sigma)).
\end{align*}
It satisfies 
\begin{align*}
\iota_* [\oO_{U,\loc}\virt] = [\oO_{U}\virt] \in K_0(U).
\end{align*}
where $\iota \colon U(\sigma) \to U$ is the inclusion, is independent of the particular choice of global resolution for $E$ and deformation invariant.
\end{defi-thm}
The purpose of this section is to generalize Definition-Theorem \ref{24} to the setting of almost perfect obstruction theories.

\medskip \subsection{Intrinsic normal cone under a cosection}
Recall from \S\ref{S4}, that under the assumptions of Definition \ref{APOT}, if $X\to S$ is equipped with an almost perfect obstruction theory $\phi$, we have the intrinsic normal sheaf 
$\fc_{X/S}$ which is a closed substack of the obstruction sheaf $Ob_X$ of $\phi$. Let $\cF=Ob_X$ in the rest of this section. 

For any local chart $Q=(U,\rho, E, r_E)$ of the sheaf stack $\cF$, we have a Cartesian square
\beq\label{25}\xymatrix{ C\ar[r] \ar[d] & E\ar[d]^{r_E}\\
\fc_{X/S}|_U\ar[r] & \cF|_U}\eeq
whose horizontal arrows are closed immersions. Hence $C$ is a closed subscheme of the vector bundle $E$ and its structure sheaf $\sO_C$ is a coherent sheaf on $E$. It is obvious that the assignment $Q\mapsto \sO_C$ is a coherent sheaf on the sheaf stack $\cF$, which we denote by  
$$\sO_{\fc_{X/S}}\in \Coh(\cF).$$

Now suppose we have a cosection 
$$\sigma:\cF=Ob_X\lra \sO_X$$
of the obstruction sheaf.
As mentioned above, it was proved in \cite{KiemLiCosection} that for any local chart 
$Q=(\rho:U\to X, r_E:E\to \cF|_U)$ of the sheaf stack $\cF$, the cone 
$C=E\times_{\cF|_U}\fc_{X/S}|_U\subset E$  in \eqref{25} has reduced support in 
\beq\label{26} E(\sigma)=E|_{U(\sigma)}\cup \ker(E\mapright{r_E} \cF|_U\mapright{\sigma} \sO_U)\eeq
where $U(\sigma)$ is the vanishing locus of $\sigma$ which is the closed subscheme of $U$ defined by the image of $\sigma|_U:\cF|_U=\rho^*\cF\to \sO_U$. 
The closed substacks $E(\sigma)$ for local charts $Q=(U,\rho,E,r_E)$ define a closed substack which we denote by $\cF(\sigma)$. We let 
$\coh(\cF)$ denote the set of isomorphism classes of coherent sheaves on the sheaf stack $\cF$ and
let 
\beq\label{29}\coh_\sigma(\cF)\eeq
denote the subset of isomorphisms classes of coherent sheaves $\aA$ on $\cF$ with support in $\cF(\sigma)$, i.e. for each local chart  
$Q=(U,\rho,E,r_E)$, $\aA_Q$ has support in $E(\sigma)$.

The assignment to a local chart $Q$ of the  ideal sheaf $I_Q$ of $E(\sigma)$ on $E$ 
is a coherent sheaf $\iI$ and there exist exact sequences
\begin{align*}
    0 \lr \iI^{j+1}\oO_{\fc_{X/S}} \lr \iI^{j}\oO_{\fc_{X/S}} \lr \iI^j\oO_{\fc_{X/S}} / \iI^{j+1}\oO_{\fc_{X/S}} \lr 0.
\end{align*}
We let 
\begin{align} \label{def of A_j}
    \aA_j = \iI^j\oO_{\fc_{X/S}} / \iI^{j+1}\oO_{\fc_{X/S}} \in \Coh(\cF)
\end{align}
so that for any local chart $Q=(U,\rho,E,r_E)$ of $\cF$, $(\aA_j)_Q\in \Coh(E)$ is a coherent sheaf supported in $E(\sigma)$. 
Note that the isomorphism class of $\aA_j$ lies in $\coh_\sigma(\cF)$ and  
\beq\label{27} [\sO_{\fc_{X/S}}]=\sum_j [\aA_j]\in K_0(\cF)\eeq
by the definition of the $\aA_j$. 

\medskip \subsection{Cosection localized Gysin maps}

In this subsection, we will define a map
\beq\label{28} 0^!_{\cF,\sigma}:\coh_\sigma(\cF)\lra K_0(X(\sigma))\eeq
where $X(\sigma)$ is the vanishing locus of the cosection $\sigma.$
The cosection localized virtual structure sheaf $[\sO_{X,\loc}\virt]\in K_0(X(\sigma))$ will be defined by 
\beq\label{30} [\sO_{X,\loc}\virt]=\sum_j 0^!_{\cF,\sigma}(\aA_j)\eeq
with $\aA_j$ from \eqref{def of A_j}. By construction, it will follow that the pushforward of $ [\sO_{X,\loc}\virt]$ by the inclusion $X(\sigma)\hookrightarrow X$ is the usual virtual structure sheaf $[\sO_X\virt]\in K_0(X).$

\medskip

Let $\tau:\tilde{X}\to X$ be the blowup of $X$ along $X(\sigma)$. Let $D$ denote the exceptional divisor.  
The cosection $\sigma:\cF\to \sO_X$ lifts to a surjection $\tau^*\cF=\cF|_{\tilde{X}}\to \sO_{\tilde{X}}(-D)$ whose kernel is denoted by $\cF'$, so that we have an exact sequence 
$$0\lra \cF'\mapright{f} \tau^*\cF\lra \sO_{\tilde{X}}(-D)\lra 0.$$

Let  $\aA\in \Coh(\cF)$ whose isomorphism class lies in $\coh_\sigma(\cF)$. 
By the pullback functors \eqref{11} and \eqref{15}, we have 
\beq\label{44} [f^*\tau^*\aA]\in K_0(\cF').\eeq 
By applying the Gysin map for $\cF'$ (Definition \ref{32}), we obtain
$$0^!_{\cF'}  [f^*\tau^*\aA]\in K_0(\tilde{X}).$$
Then we intersect it with $-D$ to obtain  
$$D^\vee\cdot 0^!_{\cF'}  [f^*\tau^*\aA]=[\sO_{\tilde{X}}\to \sO_{\tilde{X}}(D)]\otimes 0^!_{\cF'}  [f^*\tau^*\aA]\in K_0(D).$$
Now we push it down to $X(\sigma)$ by $\tau|_D$ to obtain
\beq\label{33} (\tau|_D)_* \left( D^\vee\cdot 0^!_{\cF'}  [f^*\tau^*\aA]\right) \in K_0(X(\sigma)).\eeq
Let $\imath:X(\sigma)\to X$ and $\tilde{\imath}:D\to \tilde{X}$ denote the inclusions. Then 
\beq\label{34} \imath_*(\tau|_D)_*  \left( D^\vee\cdot 0^!_{\cF'}  [f^*\tau^*\aA]\right)
=\tau_*\tilde{\imath}_*\left( D^\vee\cdot 0^!_{\cF'}  [f^*\tau^*\aA]\right)\eeq
\[=\tau_* 0^!_{\sO_{\tilde{X}}(-D)}0^!_{\cF'}  [f^*\tau^*\aA] 
=\tau_*0^!_{\tau^*\cF} [f_*f^*\tau^*\aA]\] \[=0^!_\cF\tau_* [f_*f^*\tau^*\aA]
=\sum_{i\ge 0}(-1)^i0^!_\cF [R^i\tau_*f_*f^*\tau^*\aA].\]

\medskip

Since $\aA$ is a sheaf with support in $\cF(\sigma)$ and $\tau$ is an isomorphism on $X-X(\sigma)$, we find that $R^i\tau_*f_*f^*\tau^*\aA|_{X-X(\sigma)}=0$
is a coherent sheaf on  the closed substack $\cF|_{X(\sigma)}$ for $i>0$. 
Likewise, the kernel and cokernel of the natural homomorphism
$$\eta_\aA: \aA\lra \tau_*f_*f^*\tau^*\aA$$
are coherent sheaves on $\cF|_{X(\sigma)}$.
Let 
\beq\label{37} R_{\aA}:=
 [\ker(\eta_\aA)] -[\coker(\eta_\aA)] - \sum_{i \geq 1}(-1)^i [R^i \tau_*f_*f^*\tau^*\aA] \in K_0(\cF|_{X(\sigma)}).\eeq
Then by Lemma \ref{35}, we have
 \beq\label{38}
 \imath_*0^!_{\cF|_{X(\sigma)}}R_\aA=0^!_\cF \imath_*R_\aA \eeq
\[ 
 =0^!_\cF\left([\aA]-[\tau_*f_*f^*\tau^*\aA]-\sum_{i \geq 1}(-1)^i [R^i \tau_*f_*f^*\tau^*\aA]\right)\] \[ =0^!_\cF[\aA]-\sum_{i\ge 0}(-1)^i0^!_\cF [R^i\tau_*f_*f^*\tau^*\aA]\in K_0(X).
\]

\begin{defi}\label{39}
The \emph{cosection localized Gysin map} 
\beq\label{41} 0^!_{\cF,\sigma}:\coh_\sigma(\cF)\lra K_0(X(\sigma))\eeq
for $\cF$ is defined by 
$$0^!_{\cF,\sigma}[\aA]=  (\tau|_D)_* \left( D^\vee\cdot 0^!_{\cF'}  [f^*\tau^*\aA]\right)
+0^!_{\cF|_{X(\sigma)}}R_\aA\in K_0(X(\sigma))$$
for any coherent sheaf $\aA$ on $\cF$ with support in $\cF(\sigma)$. 
\end{defi}

By adding \eqref{33} and \eqref{38}, we obtain the following comparison of the Gysin maps $0^!_\cF$ and $0^!_{\cF,\sigma}$. 
\begin{prop}\label{40} For $[\aA]\in \coh_\sigma(\cF)$, we have the equality 
$$\imath_* 0^!_{\cF,\sigma}(\aA)=0^!_\cF[\aA]\in K_0(X).$$
\end{prop}

\medskip \subsection{Cosection localized virtual structure sheaf for almost perfect obstruction theory} Using \eqref{def of A_j}, \eqref{27} and \eqref{41}, we may now
generalize the {cosection localized virtual structure sheaf} in \cite{KiemLiKTheory} to \DM stacks equipped with almost perfect obstruction theories. 

\begin{defi}\label{42}
The \emph{cosection localized virtual structure sheaf} for an almost prefect obstruction theory $\phi$ is defined by
\begin{align*}
    [\oO_{X,\loc}\virt] := \sum_{j \geq 0} 0^!_{\Ob_X, \sigma}(\aA_j) \in K_0(X(\sigma))
\end{align*}
where $0^!_{\Ob_X, \sigma}$ is the cosection localized Gysin map in Definition \ref{41} and the sheaves $\aA_j$ are defined in \eqref{def of A_j}. 
\end{defi}

\begin{prop}\label{43}
The pushforward of $[\sO_{X,\loc}\virt]$ by the inclusion $\imath:X(\sigma)\to X$ is the ordinary virtual structure sheaf $[\sO_X\virt]\in K_0(X)$ in \cite{KiemSavvas}.\end{prop}
\begin{proof}
By Proposition \ref{40}, \eqref{27} and Definition \ref{virtual structure sheaf}, we have 
$$\imath_*[\sO_{X,\loc}\virt]=\sum_{j \geq 0} \imath_*0^!_{\Ob_X, \sigma}(\aA_j)
=\sum_{j \geq 0} 0^!_{\Ob_X}[\aA_j]$$
$$=0^!_{\Ob_X}\sum_{j \geq 0} [\aA_j]=0^!_{\Ob_X}[\sO_{\fc_{X/S}}]=[\sO_X\virt]$$
as desired. 
\end{proof}

The cosection localized virtual structure sheaf is deformation invariant. The proof is rather standard and  can be found in the Appendix. 

\begin{rmk} \label{45} 
As discussed carefully in \cite{KiemLiKTheory}, we can be quite flexible in choosing a lift of $\aA$ to a class in $K_0(\cF')$. Above, we used \eqref{44} for simplicity but we could use left derived pullbacks $Lf^*$ and $L\tau^*$ instead of the ordinary pullbacks $f^*$ and $\tau^*$. With this derived choice, we have a homomorphism $$0^!_{\cF,\loc}:K_0(\cF(\sigma))\lra K_0(X(\sigma))$$
where $K_0(\cF(\sigma))$ is the Grothendieck group of the abelian category $\Coh_\sigma(\cF)$ of coherent sheaves on $\cF$ with support in $\cF(\sigma)$.     
\end{rmk}

\section{Virtual Torus Localization} \label{Torus localization section}

A virtual torus localization formula has been established at the level of intersection theory for virtual fundamental cycles in the cases of perfect \cite{GrabPand} and semi-perfect obstruction theory \cite{KiemLocalization} and  at the level of $K$-theory for virtual structure sheaves for perfect obstruction theory \cite{Qu}. In this section, we generalize the formula to the setting of virtual structure sheaves in $K$-theory obtained by an almost perfect obstruction theory. 

\medskip \subsection{$T$-equivariant almost perfect obstruction theory}

Let $T = \bC^*$ denote the one-dimensional torus and $X$ a Deligne-Mumford stack with an action of $T$. We denote the fixed locus by $F$. This is the closed substack locally defined by $\Spec A / (A^{mv})$ on an equivariant \'{e}tale chart $\Spec A \to X$, where $(A^{mv})$ denotes the ideal generated by weight spaces corresponding to non-zero $T$-weights. Finally, let
\begin{align*}
    \iota \colon F \lr X
\end{align*}
denote the inclusion map. For details on group actions on stacks, we refer the reader to \cite{Romagny}.

We can give the following definition, which generalizes directly the definition of an almost perfect obstruction theory.

\begin{defi} \emph{($T$-equivariant almost perfect obstruction theory)} \label{equivariant APOT}
Let $X$ be a Deligne-Mumford stack with a $T$-action. A \emph{$T$-equivariant almost perfect obstruction theory} $\phi$ consists of the following data:
\begin{enumerate}
    \item[(a)] A $T$-equivariant \'{e}tale covering $\lbrace X_\alpha \to X \rbrace_{\alpha \in A}$ of $X$.
    \item[(b)] For each index $\alpha \in A$, an object $E\lalp \in D([X\lalp/T])$ and a morphism $\phi_\alpha \colon E_\alpha \to \bL_{X_\alpha}$ in $D([X\lalp/T])$ which is a perfect obstruction theory on $X\lalp$.
\end{enumerate}
These are required to satisfy the following conditions:
\begin{enumerate}
\item For each pair of indices $\alpha, \beta$, there exists a $T$-equivariant isomorphism \begin{align*}
\psi_{\alpha \beta} \colon \Ob_{X_\alpha} \vert_{X_{\alpha\beta}} \lra \Ob_{X_\beta} \vert_{X_{\alpha\beta}}
\end{align*}
so that the collection $\lbrace \Ob_{X\lalp}=h^1(E_\alpha^\vee), \psi\lab \rbrace$ gives a descent datum of a sheaf $\Ob_X$, called the obstruction sheaf, on $X$.
\item For each pair of indices $\alpha, \beta$, there exists a $T$-equivariant \'{e}tale covering $\lbrace V_\lambda \to X\lab \rbrace_{\lambda \in \Gamma}$ of $X\lab=X_\alpha\times_XX_\beta$ such that for any $\lambda$, the perfect obstruction theories $\phi_\alpha \vert_{V_{\lambda}}$ and $\phi_\beta \vert_{V_{\lambda}}$ are isomorphic and compatible with $\psi\lab$. This means that there exists an isomorphism
\begin{align*}
    \eta_{\alpha\beta\lambda} \colon E_\alpha \vert_{V_{\lambda}} \lr E_\beta \vert_{V_{\lambda}}
\end{align*} 
in $D( [ V_\lambda / T])$ fitting in a commutative diagram
\begin{align} \label{loc 4.1}
    \xymatrix{
    E_\alpha \vert_{V_{\lambda}} \ar[d]_-{\phi\lalp|_{V_\lambda}} \ar[r]^-{ \eta_{\alpha\beta\lambda}} & E_\beta \vert_{V_{\lambda}} \ar[d]^-{\phi\lbet|_{V_\lambda}} \\
    \bL_{X_\alpha}|_{V_\lambda} \ar[r] \ar[dr]_\cong & \bL_{X_\beta}|_{V_\lambda} \ar[d]^\cong \\
    & \bL_{V_\lambda}
    }
\end{align}
which moreover satisfies $h^1(\eta_{\alpha\beta\lambda}^\vee) = \psi\lab^{-1}|_{V_\lambda}$.
\end{enumerate}
\end{defi}

In the above, $D([X\lalp / T])$ and $D([V\lal / T])$ denote the bounded derived categories of $T$-equivariant quasi-coherent sheaves on $U\lalp$ and $V\lal$ respectively.

\medskip \subsection{$T$-equivariant almost perfect obstruction theory on the fixed locus} Suppose that $X$ is a Deligne-Mumford stack with an action of $T$, equipped with a $T$-equivariant almost perfect obstruction theory as above. Let $F\lalp = X\lalp \times_X F$, so that $\lbrace F\lalp \to F \rbrace_{\alpha \in A}$ gives an \'{e}tale covering of the fixed locus $F$.

For each index $\alpha$, we have the decomposition 
\begin{align}
    E\lalp|_{F\lalp} = E\lalp|_{F\lalp}^{fix} \oplus E\lalp|_{F\lalp}^{mv}
\end{align}
into the $T$-fixed and moving part. Moreover, $\phi\lalp|_{F\lalp}$ similarly decomposes as a direct sum of
\begin{align} \label{loc 4.3}
    \phi\lalp^{fix} \colon E\lalp|_{F\lalp}^{fix} \to \bL_{U\lalp}|_{F\lalp}^{fix} \text{  and  }  \phi\lalp^{mv} \colon E\lalp|_{F\lalp}^{mv} \to \bL_{X\lalp}|_{F\lalp}^{mv}
\end{align}
Since $F\lalp$ has a trivial $T$-action, the morphism $\bL_{X\lalp}|_{F\lalp} \to \bL_{F\lalp}$ factors through
\begin{align} \label{loc 4.4}
    \bL_{X\lalp}|_{F\lalp}^{fix} \lr \bL_{F\lalp}
\end{align}
Composing \eqref{loc 4.3} and \eqref{loc 4.4} we obtain a morphism
\begin{align}
    \phi\lalp^F \colon E\lalp|_{F\lalp}^{fix} \lr \bL_{F\lalp}
\end{align}
By \cite{GrabPand}, this gives a perfect obstruction theory on $F\lalp$.

\begin{prop} \label{induced APOT on F}
The \'{e}tale covering $\lbrace F\lalp \to F \rbrace$ and the perfect obstruction theories $\phi\lalp^F \colon E\lalp|_{F\lalp}^{fix} \to \bL_{F\lalp}$ form an induced almost perfect obstruction theory $\phi^F$ on the fixed locus $F$ with obstruction sheaf $\Ob_{F} = \Ob_X |_{F}^{fix}$.
\end{prop}

\begin{proof}
We need to verify that conditions (1) and (2) in Definition~\ref{APOT} hold for the perfect obstruction theories on the given \'{e}tale cover of $F$. 

It is clear that $\Ob_{F\lalp} = h^1(E\lalp^\vee|_{F\lalp}^{fix}) = \Ob_{X\lalp}|_{F\lalp}^{fix}$. Since $F\lab = F\lalp \times_{F} F\lbet$ is the fixed locus of $X\lab = X\lalp \times_X X\lbet$ and $\psi\lab$ is $T$-equivariant we obtain induced isomorphisms
\begin{align}
    \psi\lab^F := \psi\lab|_{F\lab}^{fix} \colon \Ob_{F\lalp}|_{F\lab} \lr \Ob_{F\lbet}|_{F\lab}
\end{align}
which satisfy the cocycle condition and give descent data for the obstruction sheaf $\Ob_{F} = \Ob_{X}|_{F}^{fix}$.

Let $V\lal^T$ denote the fixed locus of $V\lal$. Similarly by $T$-equivariance, the isomorphisms $\eta_{\alpha \beta \lambda}$ induce isomorphisms 
$$\eta_{\alpha \beta \lambda}^F := \eta_{\alpha \beta \lambda}|_{V\lal^T}^{fix} \colon E\lalp|_{V\lal^T}^{fix} \lr E\lbet|_{V\lal^T}^{fix}$$
fitting in a commutative diagram
\begin{align} 
    \xymatrix{
    E_\alpha \vert_{V_{\lambda}^T}^{fix} \ar[d]_-{\phi\lalp^F|_{V_\lambda^T}} \ar[r]^-{ \eta_{\alpha \beta \lambda}^F} & E_\beta \vert_{V_{\lambda}^T}^{fix} \ar[d]^-{\phi\lbet^F|_{V_\lambda^T}} \\
    \bL_{F_\alpha}|_{V_\lambda^T} \ar[r] \ar[dr] & \bL_{F_\beta}|_{V_\lambda^T} \ar[d] \\
    & \bL_{V_\lambda^T}
    }
\end{align}
and satisfying $h^1 ( \eta_{\alpha\beta\lambda}^\vee|_{V\lal^T}^{fix} ) = \psi\lab^{-1}|_{F\lab}^{fix}$, as desired.
\end{proof}

Let $N\lalp\virt = (E\lalp|_{F\lalp}^{mv})^\vee$ be the virtual normal bundle of $F\lalp$ in $X\lalp$ and write $E\lalp^F =  E\lalp|_{F\lalp}^{fix}$ for brevity from now on.

In order to prove the torus localization formula in the next subsection, we will need to modify the almost perfect obstruction theory $\phi^F$ on $F$. To this end, we make the following assumption.

\begin{assum} \label{56}
There exists a two-term complex 
\begin{align*}
    N\virt = [N_0 \lr N_1]
\end{align*}
of locally free sheaves on $F$
and an isomorphism $\mu \colon h^1(N\virt) \to \Ob_\phi|_{F}^{mv}$ such that for any index $\alpha$ we have an isomorphism $N\virt|_{F\lalp} \cong N\lalp\virt$ whose homology in degree $1$ induces the restriction $\mu|_{F\lalp}$. We write $N^{-1}=N_1^\vee, N^0 = N_0^\vee$.
\end{assum}

This assumption may turn out to be unnecessary in the future but under the current state of technology, this is a weakest assumption for a proof of the virtual torus localization formula, Theorem \ref{46} below.   

We will compare the virtual structure sheaves $[\sO_X\virt]$ on $X$ and $[\sO_F\virt]$ on $F$ through an intermediate virtual structure sheaf $[\tilde\sO_F\virt]$, after introducing an auxiliary almost perfect obstruction theory on $F$. Here is an outline: 
\begin{enumerate}
\item (Proposition \ref{47}) $[\tilde\sO_F\virt]=[\sO_F\virt]\cap e(N_1)$ where $(\cdot)\cap e(N_1)$ is tensoring the Koszul complex $\wedge^\bullet N^{-1}$ for the zero section of $N_1$.
\item (Proposition \ref{48}) $\imath^![\sO_X\virt]=[\tilde\sO_F\virt]$ where $\imath:F\to X$ is the inclusion and $\imath^!$ is the virtual pullback defined in \cite{Qu}.
\item (Proposition \ref{49}) $\imath_*\xi=[\sO_X\virt]$ for some $\xi\in K_0(F)\otimes_\QQ \QQ(t)$.
\item (Theorem \ref{46}) Since $\xi\cap e(N_0)=\imath^!\imath_*\xi=\imath^![\sO_X\virt]=[\sO_F\virt]\cap e(N_1)$, we have $\xi=[\sO_F\virt]/e(N\virt)$ where $e(N\virt)=e(N_0)/e(N_1).$
 We thus obtain the torus localization formula
\beq\label{50} [\sO_X\virt]=\imath_*\frac{[\sO_F\virt]}{e(N\virt)}.\eeq
\end{enumerate}
In the subsequent subsections, we will work out the details of the outline. 

\medskip \subsection{An auxiliary almost perfect obstruction theory on the fixed locus}

We will introduce a new almost perfect obstruction theory on $F$ by adding the locally free sheaf $N_1$ to the obstruction sheaf and compare the virtual structure sheaves arising from the old and new almost perfect obstruction theories. 

For each index $\alpha$, we let
\begin{align} \label{enlarged POT}
    \tE\lalp^F = E\lalp^F \oplus N_1^\vee|_{F\lalp} [1] \text{  and  } \tphi\lalp^F \colon \tE\lalp^F \lr E\lalp^F \xrightarrow{\phi\lalp^F} \bL_{F\lalp},
\end{align}
where the first arrow $\tE\lalp^F \to E\lalp^F$ in the composition is projection onto the first summand. It is clear that $\tphi\lalp^F$ is a perfect obstruction theory on $F\lalp$ with obstruction sheaf $\Ob_{\tphi\lalp^F} = h^1( (\tE\lalp^F)^\vee ) = \Ob_{\phi\lalp^F} \oplus N_1|_{F\lalp}$.

\begin{prop} \label{enlarged APOT on F}
The \'{e}tale covering $\lbrace F\lalp \to F \rbrace$ and the perfect obstruction theories $\tphi\lalp^F \colon \tE\lalp^F \to \bL_{F\lalp}$ form an almost perfect obstruction theory $\tphi^F$ on $F$ with obstruction sheaf $\widetilde\Ob_{F} = \Ob_{F} \oplus N_1$.
\end{prop}

\begin{proof}
The proof is identical to that of Proposition~\ref{induced APOT on F} using $$\tilde{\psi}\lab^F = \psi\lab^F \oplus \id_{N_1|_{F\lab}} \quad \text{and}\quad \tilde{\eta}_{\alpha \beta \lambda}^F = \eta_{\alpha \beta \lambda}^F \oplus \id_{N_1|_{V\lal^T}}$$ for the appropriate compatibilities.
\end{proof}

The two almost perfect obstruction theories $\phi^F$ and $\tphi^F$ on $F$ induce virtual structure sheaves $[\oO_F\virt] \in K_0(F)$ and $[\tilde\oO_F\virt] \in K_0(F)$ respectively. These are related by the following formula.

\begin{prop} \label{47}
$[\tilde \oO_F\virt] = [\oO_F\virt]\cap e(N_1)\in K_0(F)$ 
 where $(\cdot)\cap e(N_1)$ denotes tensoring the Koszul complex $\wedge^\bullet (N^{-1})$ for the zero section of $N_1$. 
\end{prop}

\begin{proof}
Let $j_{\phi^F} \colon \fc_F \to \Ob_{F}$ be the induced embedding of the coarse intrinsic normal cone stack of $F$ into the obstruction sheaf stack of $\phi^F$. By the definition of $\tphi^F$, it is easy to see that the embedding $j_{\tphi^F}$ is the composition of $j_{\phi^F}$ with the inclusion $\Ob_{F} \to \widetilde\Ob_{F} = \Ob_{F} \oplus N_1$ as the first summand.

By Definition~\ref{virtual structure sheaf}, we have
$$[\tilde\oO_F\virt] = 0_{\widetilde{\Ob}_{F}}^![\oO_{\fc_F}] = 0_{\Ob_{F} \oplus N_1}^![\oO_{\fc_F}]$$
$$=0^!_{N_1}\left(0^!_{Ob_F}[\sO_{\fc_F}]\right)=0^!_{N_1}[\sO_F\virt]=[\sO_F\virt]\cap e(N_1)$$
as desired. 
\end{proof}

\medskip \subsection{Refined intersection with the fixed locus}
In this subsection, we prove Proposition \ref{48} below.

\begin{lem} \label{lemma 4.6}
$[\tilde\oO_F\virt] = 0_{\Ob_X|_F \oplus N_0}^! [\oO_{\fc_{F/\cC_X}}] \in K_0(F)$.
\end{lem}

\begin{proof} The proof is an adaptation of a standard functoriality argument, following the lines of the proof of \cite[Theorem~4.3]{KiemSavvas}. We repeat a sketch of the argument here for the convenience of the reader.

Let $\mM_{X}^\circ \to \bP^1$ be the deformation of $\mathrm{Spec}\,\CC$ to the intrinsic normal cone stack $\cC_{X}$. Let $\wW = \mM_{F \times \bP^1 / \mM_{X}^\circ}^\circ$ be the double deformation space given by the deformation of $F \times \bP^1$ inside $\mM_{X}^\circ$ to its normal cone $\cC_{F \times \bP^1 / \mM_{X}^\circ}$. We also write $\nN_{F \times \bP^1 / \mM_X^\circ}$ for the intrinsic normal sheaf with coarse moduli sheaf the sheaf stack $\fn_{F \times \bP^1 / \mM_X^\circ}$. We have a morphism $\wW \to \bP^1 \times \bP^1$ and denote the two projections $\wW \to \bP^1$ by $\pi_1$ and $\pi_2$ respectively.

The fiber over $(1,0)$ is $\cC_F$, while the flat specialization at the point $(0,0)$ along $\lbrace 0 \rbrace \times \bP^1$ is $\cC_{F/\cC_X}$. 
In particular, the flat specialization at $(0,0)$ along $\bP^1 \times \lbrace 0 \rbrace$ is also $\cC_{F / \cC_X}$ meaning that there exists a closed substack $\zZ \subset \cC_{F \times \bP^1/ \mM_X^\circ}$, flat over $\bP^1$ with fibers
\begin{align}
    \zZ_t = \begin{cases}
    \cC_{F}, & t \neq 0 \\
    \cC_{F/\cC_X}, & t = 0
    \end{cases}
\end{align}
Thus
\begin{align*} 
    [\oO_\zZ \otimes_{\oO_{\bP^1}}^L \bC_0] & = [\oO_{\zZ_0} ] = [\oO_{\cC_{F/\cC_X}}] \\ 
    [\oO_\zZ \otimes_{\oO_{\bP^1}}^L \bC_1] & = [\oO_{\zZ_1} ] = [\oO_{\cC_F}],
\end{align*}
and since $[\bC_0] = [\bC_1] \in K_0(\bP^1)$ we obtain
\begin{align*}
    [\oO_{\cC_F}] = [\oO_{\cC_{F/\cC_X}}] \in K_0(\cC_{F \times\bP^1 / \mM_X^\circ}).
\end{align*}
Pushing forward to $\nN_{F \times\bP^1 / \mM_X^\circ}$, the equality holds in $K_0(\nN_{F \times\bP^1 / \mM_X^\circ})$ as well.

The same argument at the level of coarse moduli sheaves yields the equality
\begin{align} \label{loc 4.11}
    [\oO_{\fc_F}] = [\oO_{\fc_{F/\cC_X}}] \in K_0(\fn_{F \times\bP^1 / \mM_X^\circ}).
\end{align}

As in \cite{KimKreschPant}, for each $\alpha$ have a commutative diagram of exact triangles on $F\lalp \times \bP^1$
\begin{align*}
    \xymatrix{
    E\lalp|_{F\lalp}(-1) \ar[d] \ar[r]^-{\kappa\lalp} & E\lalp|_{F\lalp} \oplus \tE\lalp^F \ar[r] \ar[d] & c(\kappa\lalp) \ar[d] \ar[r] &\\
    \bL_{X\lalp}|_{F\lalp}(-1) \ar[r]_-{\mu\lalp} & \bL_{X\lalp}|_{F\lalp} \oplus \bL_{F\lalp} \ar[r] & c(\mu\lalp) \ar[r] &
    }
\end{align*}
where $\kappa\lalp = ( z_0 \cdot \id, z_1 \cdot g\lalp)$ with $z_0, z_1$ homogeneous coordinates on $\bP^1$ and $g\lalp$ is the canonical morphism from the inclusion $F_\alpha\to U_\alpha$.
The morphism 
$\mu\lalp$ is the restriction to $F\lalp$ of a global arrow $\mu$ and it is shown in \cite{KimKreschPant} that $h^1/h^0( c(\mu)^\vee ) = \nN_{F \times \bP^1 / \mM_X^\circ}$.

By the compatibilities afforded by the almost perfect obstruction theories and the definition~\eqref{enlarged POT} of $\tE\lalp^F$, we see that the closed embeddings
\begin{align*}
    h^1(c(\mu\lalp)^\vee) \lr h^1(c(\kappa\lalp)^\vee)
\end{align*}
glue to a global embedding of sheaf stacks on $X \times \bP^1$
\begin{align*}
    \fn_{F \times \bP^1 / \mM_X^\circ} \lr \kK
\end{align*}
Moreover, it is routine to check that the fiber of $\kK$ over $0 \in \bP^1$ is $\Ob_X|_F \oplus N_0$, while the fiber at $1 \in \bP^1$ is $\widetilde{\Ob}_{F}$.
Therefore, \eqref{loc 4.11} and the discussion preceding it imply that
\begin{align*}
    [\tilde\oO_F\virt] = 0_{\widetilde{\Ob}_{F}}^! [\oO_{\fc_F}] = 0_{\Ob_X|_F \oplus N_0}^! [\oO_{\fc_{F/\cC_X}}],
\end{align*}
as desired.
\end{proof}

By the definition of $\tE\lalp^F$, for any index $\alpha$, we have a commutative diagram of exact triangles
\begin{align} \label{loc 4.12}
\xymatrix{
E\lalp |_{F\lalp} \ar[d]^{\phi\lalp|_{F\lalp}} \ar[r]^-{g\lalp} & \tE\lalp^F \ar[d]^{\tphi\lalp^F} \ar[r] & N_0^\vee |_{F\lalp} [1] \ar[d]^-{\theta\lalp} \ar[r] & \\
\bL_{X\lalp} |_{F\lalp} \ar[r] & \bL_{F\lalp} \ar[r] & \bL_{F\lalp / X\lalp} \ar[r] &
}
\end{align}

By the following proposition, $\theta\lalp$ gives a perfect obstruction theory on the morphism $F\lalp \to U\lalp$.
\begin{prop}
The \'{e}tale covering $\lbrace F\lalp \to F \rbrace$ and the perfect obstruction theories $\theta\lalp \colon N_0^\vee |_{F\lalp}[1] \to \bL_{F\lalp/X\lalp}$ form an almost perfect obstruction theory $\theta$ on $\iota \colon F \to X$ with obstruction sheaf $\Ob_{F/X} = N_0$. 
\end{prop}

\begin{proof}
Since $F\lalp \to X\lalp$ is a closed embedding, $\bL_{F\lalp / X\lalp}$ is supported in degree $-1$. The long exact sequence in cohomology for the diagram~\eqref{loc 4.12} yields
\small
\begin{align*}
    \xymatrix{
    h^{-1}(E\lalp |_{F\lalp}) \ar[d]^{h^{-1}(\phi\lalp|_{F\lalp})} \ar[r] & h^{-1}(\tE\lalp^F) \ar[d]^{h^{-1}(\tphi\lalp^F)} \ar[r] & N_0^\vee |_{F\lalp} \ar[d]^-{h^{-1}(\theta\lalp)} \ar[r] & h^{0}(E\lalp |_{F\lalp}) \ar[d]^{h^{0}(\phi\lalp|_{F\lalp})} \ar[r] & h^{0}(\tE\lalp^F) \ar[d]^{h^{0}(\tphi\lalp^F)}\\
    h^{-1}(\bL_{X\lalp} |_{F\lalp}) \ar[r] & h^{-1}(\bL_{F\lalp} )\ar[r] & h^{-1}(\bL_{F\lalp / X\lalp}) \ar[r] & h^{0}(\bL_{X\lalp} |_{F\lalp}) \ar[r] & h^{0}(\bL_{F\lalp} )
    }
\end{align*}\normalsize

The two leftmost vertical arrows are surjections, while the two rightmost arrows are isomorphisms. Thus, by the five lemma, the middle arrow $h^{-1}(\theta\lalp)$ is a surjection. 
Since $h^0(N_0^\vee |_{F\lalp}[1]) = h^0(\bL_{F\lalp / U\lalp}) = 0$, $\theta\lalp$ is a perfect obstruction theory for the inclusion $F_\alpha\to X\lalp$.

The rest of the proof is a diagram chase using the definition~\eqref{enlarged POT} of $\tE\lalp^F$ and the compatibilities $\psi\lab$ from Definition~\ref{APOT} and $\tilde{\eta}_{\alpha \beta \lambda}^F$ from Proposition~\ref{enlarged APOT on F}. The details are omitted.
\end{proof}

The almost perfect obstruction theory $\theta$ induces a closed embedding
\begin{align*}
    j \colon \fc_{F/X} \lr \Ob_{F/X} = N_0
\end{align*}
Since $\iota \colon F \to X$ is an embedding, the coarse intrinsic normal cone stack $\fc_{F/X}$ coincides with the intrinsic normal cone $\cC_{F/X}$ and we obtain a virtual pullback by the formula
\begin{align*}
    \imath^! \colon K_0(X) \xrightarrow{\sigma_\imath} K_0(\cC_{F/X}) \xrightarrow{j_*} K_0(N_0) \xrightarrow{0_{N_0}^!} K_0(F),
\end{align*}
where $\sigma_\imath$ is the deformation to the normal cone (cf. \cite[\S2.1]{Qu}).

Now we can prove the following.
\begin{prop}\label{48} With the notation above, we have 
$$\iota^![\sO_X\virt]=[\sO_F\virt]\cap e(N_1).$$ 
\end{prop}

\begin{proof}
By Lemma \ref{lemma 4.6} and Proposition \ref{46}, it suffices to prove
\beq\label{51}\imath^! [\oO_X\virt] = 0_{\Ob_X|_F \oplus N_0}^! [\oO_{\fc_{F/\cC_X}}] \in K_0(F).\eeq

For each local chart $Q=(U,\rho, E, r_E)$ of $Ob_X$, let $F_U=F\times_XU$, $C=\fc_X|_U\times_{Ob_X|_U}E$ be the lift of the normal cone $\fc_X$ to $E$. 

For a closed immersion $Z\hookrightarrow W$, $M_{Z/W}^\circ$ denote the deformation to the normal cone, i.e. it is the blowup of $W\times \PP^1$ along $Z\times\{0\}$ with the strict transform of $W\times \{0\}$ deleted. We have a flat morphism $M_{Z/W}^\circ\to \PP^1$ whose fiber over $t\ne 0$ (resp. $t=0$) is $W$ (resp. the normal cone $C_{Z/W}$).

Then we have a commutative diagram
\small
\beq\label{53} \xymatrix{
C\ar@{^(->}[d] & C\times \AA^1 \ar[l]\ar@{^(->}[r] \ar@{^(->}[d] & M^\circ_{F_U/C}\ar@{^(->}[d] & C_{F_U/C}\ar@{_(->}[l]\ar@{^(->}[d] \\
E\ar@{->>}[d] & E\times \AA^1 \ar[l]_-{pr_1}\ar@{^(->}[r] \ar@{->>}[d] & M^\circ_{F_U/E}\ar@{->}[d] & C_{F_U/E}\ar@{_(->}[l]\ar@{->>}[d] \ar@{^(->}[r] & E|_{F_U}\oplus N_0|_{F_U}\ar@{->>}[d]\\
Ob_X|_U\ar[d] & Ob_X|_U\times \AA^1\ar[l]\ar[d]\ar[r] &Ob_X|_U\times_UM_{F_U/U}^\circ\ar[d]&Ob_X|_{F_U}\times_{F_U} C_{F_U/U}\ar[l] \ar@{^(->}[r] \ar[d] &
Ob_X|_{F_U}\oplus N_0|_{F_U}\ar[d]\\
U& U\times \AA^1\ar[l]_-{pr_1} \ar@{^(->}[r]  & M^\circ_{F_U/U} & C_{F_U/U}\ar@{_(->}[l] \ar@{^(->}[r] & N_0|_{F_U}.
}\eeq
\normalsize
By descent, we have a diagram
\small
\beq\label{54}\xymatrix{
K_0(X)\ar[r]^{pr_1^*} & K_0(X\times \AA^1) & K_0(M^\circ_{F/X})\ar[l]\ar[r]& \\
K_0(Ob_X)\ar[r]^{pr_1^*}\ar[u]_{0^!_{Ob_X}} & K_0(Ob_X|_{X\times\AA^1})\ar[u]_{0^!_{Ob_X|_{X\times\AA^1}}}
& K_0(Ob_X|_{M^\circ_{F/X}})\ar[u]_{0^!_{Ob_X|_{M^\circ_{F/X}}}}\ar[l]\ar[r] &
}\eeq
\[\xymatrix{
\ar[r]& K_0(C_{F/X})\ar[r]^{j_*} &K_0(N_0)\ar[r]^{0_{N_0}^!} & K_0(F)\\
\ar[r]& K_0(Ob_X|_F\times_F C_{F/X})\ar[r]\ar[u]_{0^!_{Ob_X|_{C_{F/X}}}} & K_0(Ob_X|_F\oplus N_0)\ar[u]^{0^!_{Ob_X|_{N_0}}}\ar[ur]_{0^!_{Ob_X|_F\oplus N_0}}
}\]
\normalsize
which is commutative by Lemmas \ref{4}, \ref{35} and \ref{14}.
Although there are left arrows in \eqref{54}, the composition of the horizontal arrows for the top row is well defined. For the bottom row, the composition of the horizontal arrows evaluated at $[\sO_{\fc_X}]$ is well defined and equal to $[\sO_{\fc_{F/\cC_X}}]$ by \eqref{53}. By the definition of $\imath^!$, $0^!_{Ob_X}[\sO_{\fc_X}]=[\sO_X\virt]$ is mapped to $\imath^![\sO_X\virt]$. Hence \eqref{51} follows from \eqref{54}. 
\end{proof}

Finally, the following standard equality for virtual pullbacks holds in our setting (cf. \cite[Proposition~2.14]{Qu}).

\begin{prop} \label{proposition 4.10}
For any $\aA \in K_0(F)$, $\imath^! \imath_* (\aA) = \aA\cap e(N_0)=\aA \cdot \wedge^\bullet(N^0)$.
\end{prop}

\begin{proof}
This is an easy computation, similar to the proof of Proposition~\ref{47} (but simpler), using the fact that $\sigma_\iota(\aA) = \aA \in K_0(\cC_{F/X})$.  
\end{proof}

\medskip \subsection{Virtual torus localization formula} 
We are now ready to prove the virtual torus localization formula, using the results of the previous subsections.

All $K$-groups are now considered with $\bQ$-coefficients. We denote by $t$ the equivariant parameter so that $K_0^T(\Spec \bC) = \bQ[t,t^{-1}]$. 

We introduce some terminology.

\begin{defi}
Let $X$ be a Deligne-Mumford stack with an action of $T$ and $\imath \colon F \to X$ denote its $T$-fixed locus. We say that $X$ is \emph{admissible} for torus localization if the homomorphism $$\imath_* \colon K_0^T(F) \to K_0^T(X)$$ of $K_0^T(\Spec \bC) = \bQ[t,t^{-1}]$-modules becomes an isomorphism after tensoring with $\bQ(t)$.
\end{defi}

\begin{rmk}
By \cite{EdidinGraham}, if $X$ is an algebraic space with a $T$-action, then $X$ is admissible.
\end{rmk}

The following proposition shows that condition (a) in Definition~\ref{equivariant APOT} implies that Deligne-Mumford stacks with equivariant almost perfect obstruction theories are admissible.

\begin{prop} \label{49} 
Let $X$ be a Deligne-Mumford stack with an action of $T$ and a $T$-equivariant \'{e}tale atlas $f \colon U \to X$ of finite type. Then $X$ is admissible. 
\end{prop}

\begin{proof}
Let $M = X - F$ and $V = f^{-1}(M)$. Using the existence of the equivariant atlas $f \colon U \to X$, we have the excision exact sequence for $K$-theory of Deligne-Mumford stacks (cf. \cite[Proposition~3.3]{ToenRR}), and thus it suffices to show that 
\begin{align*}
    K_\ast^T(M) \otimes_{\bQ[t,t^{-1}]} \bQ(t) = 0
\end{align*}
or equivalently
\begin{align} \label{loc 4.19a}
    K_\ast^T(M) \otimes_{\bQ[t,t^{-1}]} \bC(t) = 0,
\end{align}
where $K_\ast^T(M)$ is the direct sum of all $T$-equivariant $K$-theory groups.

More generally, we will show that if $M$ is a Deligne-Mumford stack with a $T$-action and $T$-equivariant atlas $f \colon V \to M$ with empty fixed locus $M^T = \emptyset$, then \eqref{loc 4.19a} holds. In the rest of the proof, all $K$-groups are considered with $\bC$-coefficients.

Let $M_1 \sub M$ be an open substack such that $f \colon V_1 = f^{-1}(M_1) \to M_1$ is \'{e}tale of (maximal) degree $n$. $M_1$ is preserved by the $T$-action. Then, as in \cite[Proposition 4.5.5.(iii)]{KreschCycles}, we have that $M_1$ is isomorphic to the quotient of the complement of all diagonals in the $n$-fold product $V_1 \times_{M_1} ... \times_{M_1} V_1$ by the action of the symmetric group $S_n$. Thus $M_1$ is of the form $[W_1 / G_1]$ where $G_1$ is a finite group acting on a scheme $W_1$ and the $T$-action is given by an action on $W_1$ commuting with the $G_1$-action and such that $W_1^T = \emptyset$.

We write $\dim \colon R(G_1) \to \bC$ for the morphism induced by mapping every representation to its dimension with kernel the augmentation (maximal) ideal $\fm_{G_1}$. 

We thus have
$$K_\ast^T(M_1) = K_\ast^T([W_1/G_1]) = K_\ast^{T \times G_1}(W_1).$$

Since $W_1^T = \emptyset$, $T$ acts on $W_1$ with finite stabilizers and thus we may choose $h_1 \in T$ such that $W_1^{h_1} = \emptyset$. Then $h_1$ belongs to the center of $T \times G_1$ and so its conjugacy class is just $\lbrace h_1 \rbrace$. Therefore, \cite[Theorem~3.3(a)]{EdidinGraham} implies that
$$K_\ast^{T \times G_1}(W_1)_{\fm_{h_1}} = 0$$
where $\fm_{h_1}$ is the maximal ideal of the complex representation ring $R(T \times G_1) = R(G_1)[t,t^{-1}]$ generated by the augmentation ideal $\fm_{G_1}$ of $G_1$ and $(t-h_1)$. 

In particular, for any element $[x] \in K_\ast^{T \times G_1}(W_1)$ there exists a Laurent polynomial $\mu_1(t) \in R(G_1)[t,t^{-1}] - \fm_{h_1}$ such that
$$\mu_1(t) \cdot [x] = 0 \in K_\ast^{T \times G_1}(W_1).$$
Letting $\ell_1 (t) = \mu_1 (t) \mod \fm_{G_1} \in \bC[t,t^{-1}]$, we have $\ell_1(t) \neq 0$ and 
$$\ell_1(t) \cdot [x] \in \fm_{G_1} K_\ast^{T \times G_1} (W_1)$$
so that
$$\ell_1(t) \cdot [x] = \sum m_i [x_i]$$
with $m_i \in \fm_{G_1}$ and $[x_i] \in K_\ast^{T \times G_1} (W_1)$.

Repeating the argument for the $[x_i]$, we can find a non-zero Laurent polynomial $\ell_2(t) \in \bC[t,t^{-1}]$ such that
$$\ell_2(t) \ell_1(t) \cdot [x] \in \fm_{G_1}^2 K_\ast^{T \times G_1} (W_1)$$
and continuing inductively, for any positive integer $N$, non-zero Laurent polynomials 
$\ell_1(t), \ell_2(t), ..., \ell_N(t) \in \bC[t,t^{-1}]$ such that
$$\ell_N(t) ... \ell_1(t) \cdot [x] \in \fm_{G_1}^N K_\ast^{T \times G_1} (W_1).$$

Since $R(G_1)$ is an Artinian $\bC$-algebra of finite type, we have $\fm_{G_1}^N = 0$ for large enough $N$. We conclude that for any $[x] \in K_\ast^{T\times G_1} (W_1)$ there exists a non-zero Laurent polynomial $\ell(t) \in \bC[t,t^{-1}]$ such that
$$\ell(t) \cdot [x] = 0 \in K_\ast^{T\times G_1} (W_1)$$
which implies that
$$K_\ast^T(M_1) \otimes \bC(t) = K_\ast^{T\times G_1} (W_1) \otimes \bC(t) = 0$$

Letting $Z_1 = M - M_1$ to be the complement of $M_1$, by excision we are reduced to showing that $K_\ast^T(Z_1) \otimes \bC(t) = 0$. Repeating the above argument and using Noetherian induction concludes the proof.
\end{proof}

\begin{thm}\label{46} \emph{(Virtual torus localization formula)} Let $X$ be a Deligne-Mumford stack with an action of $T$ and a $T$-equivariant almost perfect obstruction theory $\phi$ such that Assumption \ref{56} holds. Let $F$ denote the $T$-fixed locus of $X$ and $\phi^F$ its induced almost perfect obstruction theory. Then
\begin{align*}
    [\oO_X\virt] = \iota_* \frac{[\oO_F\virt]}{e(N\virt)} \in  K_0^T(X) \otimes_{\bQ[t,t^{-1}]} \bQ(t)
\end{align*}
where the Euler class $e(N\virt)=e(N_0)/e(N_1)=\Lambda_{-1}^T((N\virt)^\vee)$ denotes multiplication by $\wedge^\bullet N^0/\wedge^\bullet N^{-1}$. 
\end{thm}

\begin{proof}
By Propositions \ref{48} and \ref{proposition 4.10}, we have
\begin{align} \label{loc 4.21}
    \imath^! \imath_*  \frac{[\oO_F\virt]}{e(N\virt)} = \imath^! \imath_* \left( \frac{[\oO_F\virt]\cap e(N_1)}{e(N_0)} \right) = [\oO_F\virt]\cap e(N_1)=\imath^![\sO_X\virt]
\end{align}
By the previous proposition, $\imath_*$ becomes an isomorphism after tensoring with $\bQ(t)$. Hence,  Proposition~\ref{proposition 4.10} implies that $\frac{\imath^!}{e(N_0)}$ is the inverse of $\imath_*$. In particular, $\imath^!$ is injective. Thus the virtual torus localization formula follows from \eqref{loc 4.21}.
\end{proof}

\section{Torus Localization of Cosection Localized Virtual Structure Sheaf} \label{combination section}

Torus localization for cosection localized virtual cycles has been established in \cite{ChangKiemLi} and in \cite{KiemLocalization} in the settings of perfect and semi-perfect obstruction theory respectively. In this section, we prove the corresponding statement for virtual structure sheaves obtained by an almost perfect obstruction theory. As usual, $T$ denotes the torus $\bC^\ast$.

Let $X$ be a Deligne-Mumford stack with a $T$-action and a $T$-equivariant almost perfect obstruction theory $\phi$ given by perfect obstruction theories $\phi\lalp \colon E\lalp \to \bL_{X\lalp}$ on a $T$-equivariant \'{e}tale cover $\lbrace \rho\lalp \colon X\lalp \to X \rbrace$. Moreover, suppose that we have a $T$-invariant cosection
\begin{align*}
    \sigma \colon \Ob_X \lr \oO_X.
\end{align*}

We use the same notation as in \S\ref{Scos} and \S\ref{Torus localization section}. Let $F(\sigma) = F \times_X X(\sigma)$.

By Proposition~\ref{induced APOT on F}, the fixed locus $F$ admits an almost perfect obstruction theory $\phi^F$ on the \'{e}tale cover $\lbrace F\lalp \to F \rbrace$ with obstruction sheaf $$\Ob_{F} = \Ob_X |_F^{fix}.$$
Since the cosection $\sigma$ is $T$-invariant, $\sigma|_F$ factors through a morphism
\begin{align*}
    \sigma_F \colon \Ob_{F} \lr \oO_F
\end{align*}
whose zero locus is precisely $F(\sigma)$.
Therefore, by Definition \ref{42}, we have cosection localized virtual structure sheaves
\begin{align*}
    [\oO_{X,\loc}\virt] \in K_0(X(\sigma)), \quad  [\oO_{F,\loc}\virt] \in K_0(F(\sigma)).
\end{align*}

\begin{thm} \label{55}
Let $\imath \colon F(\sigma) \to X(\sigma)$ denote the inclusion and suppose that Assumption \ref{56} holds. Then
\begin{align*}
    [\oO_{X,\loc}\virt] = \iota_*  \frac{[\oO_{F,\loc}\virt]}{e(N\virt)}  \in K_0^T(X(\sigma)) \otimes_{\bQ[t,t^{-1}]} \bQ(t).
\end{align*}
\end{thm}

\begin{proof}
The proof proceeds along the same steps of the proof of Theorem \ref{46} with minor modifications to account for the presence of the cosections $\sigma$ and $\sigma_F$.

The proof of Proposition \ref{47} goes through, using the cosection 
$$\tilde{\sigma}_F \colon \widetilde{\Ob}_{F} = \Ob_{F} \oplus N_1 \lr \Ob_{F} \xrightarrow{\sigma_F} \oO_F$$
and working with the sheaves $\aA_j$ defined in \eqref{def of A_j}. 

For Lemma~\ref{lemma 4.6}, it is shown in \cite[Section~5]{KiemLiCosection} that there exists an extended cosection $\bar{\sigma}$ for $\kK$ which restricts to the cosection $\tilde{\sigma}^F$ over $1 \in \bP^1$ and the cosection
$$\Ob_X |_F \oplus N_0 \lr \Ob_X|_F \xrightarrow{\sigma|_F} \oO_F$$
over $0 \in \bP^1$. Furthermore, it is shown that the coarse intrinsic normal cone $\fc_{F \times \bP^1 / \mM_X^\circ}$ has reduced support in $\kK(\bar{\sigma})$. A more detailed account of the same argument is also given in Appendix~\ref{appendix}. With these considerations, the proof goes through identically.

Finally, in Proposition \ref{48}, it is easy to check that all the sheaves and cones have the appropriate supports with respect to the cosections at hand.
\end{proof}

\section{Applications} \label{applications section}

We now discuss some applications of the theory developed thus far. We establish a wall-crossing formula for simple $\bC^\ast$-wall crossing in the setting of almost perfect obstruction theory, using the torus localization formula of Theorem \ref{46}. Moreover, we show that the Jiang-Thomas dual obstruction cone in  \cite{JiangThomas} admits an almost perfect obstruction theory and thus gives rise to $K$-theoretic invariants, using the combination of torus localization and cosection localization in Theorem \ref{55}.

\medskip \subsection{$K$-theoretic simple $\bC^\ast$-wall crossing} In this subsection, we establish a $K$-theoretic wall crossing formula for simple $\bC^\ast$-wall crossing, following the construction given in \cite{KiemLi}.

Let $X$ be a Deligne-Mumford stack acted on by the torus $T = \bC^\ast$ and equipped with a $T$-equivariant almost perfect obstruction theory $\phi$ consisting of perfect obstruction theories $\phi\lalp \colon E\lalp \to \bL_{X\lalp}$ on a $T$-equivariant \'{e}tale cover $\lbrace X\lalp \to X\rbrace$. Let $F$ denote the fixed locus.
We assume that Assumption \ref{56} holds, so that we have the two-term complex $N\virt = [ N_0 \to N_1]$ of locally free sheaves on $F$.
Let
\begin{enumerate}
    \item $X^s$ be the open substack of $X$, consisting of $x \in X$ such that the orbit $T \cdot x$ is $1$-dimensional and closed in $X$;
    \item $\Sigma_\pm^0 = \lbrace x \in X - \left( X^s \cup F \right) \ \vert \ \lim_{t \to 0} t^{\pm 1} \cdot x \in F \rbrace$;
    \item $\Sigma_\pm = \Sigma_\pm^0 \cup F$;
    \item $X_\pm = X - \Sigma_\mp \sub X$;
    \item $M_\pm = [ X_\pm / T ] \sub M = [X / T]$.
\end{enumerate}
We assume that $M_\pm$ are separated Deligne-Mumford stacks.

The master space associated to the wall crossing $M_\pm$ is defined by
\begin{align} \label{master space}
    \mM = [X \times \bP^1 - \Sigma_- \times \lbrace 0 \rbrace - \Sigma_+ \times \lbrace \infty \rbrace / \bC^\ast]
\end{align}
where $\bC^\ast$ acts trivially on $X$ and on $\bP^1$ by $t \cdot (a : b) = (a : tb)$. $\mM$ admits an \'{e}tale cover $\lbrace \mM\lalp \to \mM \rbrace$ where 
\begin{align*}
    \mM\lalp = [X\lalp \times \bP^1 - ( \Sigma_- \times_X X\lalp ) \times \lbrace 0 \rbrace - (\Sigma_+ \times_X X\lalp) \times \lbrace \infty \rbrace / \bC^\ast].
\end{align*}
The $T$-action on $X$ induces an action of $T$ on $\mM$ with fixed locus
\begin{align*}
    M_+ \sqcup F \sqcup M_-.
\end{align*}

For each index $\alpha$, the pullback of $\phi\lalp$ to $X\lalp \times \bP^1$ is $\bC^\ast$-equivariant and therefore by descent we obtain a morphism $\bar{\phi}\lalp \colon \bar{E}\lalp \to \bL_{\mM\lalp}$.

\begin{prop}
The morphisms $\bar{\phi}\lalp \colon \bar{E}\lalp \to \bL_{\mM\lalp}$ on the \'{e}tale cover $\lbrace \mM\lalp \to \mM \rbrace$ form a $T$-equivariant almost perfect obstruction theory $\bar{\phi}$ on $\mM$.
\end{prop}

\begin{proof}
The proof is straightforward, similar to the arguments given for checking the axioms of almost perfect obstruction theories in \S \ref{Torus localization section}.
\end{proof}

Applying the virtual torus localization formula then yields the following theorem.

\begin{thm} \label{theorem 7.2}
With the above notation and conditions, we have
\begin{align*}
    [\oO_{M_+}\virt] - [\oO_{M_-}\virt] = \mathrm{res}_{t=1}  \frac{[\oO_F\virt]}{e(N\virt)} \in K_0^T(\mM).
\end{align*}
\end{thm}
\begin{proof}
By Theorem \ref{46}, we have
\begin{align} \label{loc 6.2}
    [\oO_\mM\virt] = \frac{[\oO_{M_+}\virt]}{1-t} + \frac{[\oO_{M_-}\virt]}{1-t^{-1}} + \frac{[\oO_F\virt]}{e(N\virt)}
\end{align}
since, by construction, the normal bundle of $M_+$ is trivial with $T$-weight $1$ and the normal bundle of $M_-$ is trivial with $T$-weight $-1$.

Since $[\oO_\mM\virt] \in K_0^T(\mM)$, it has zero residue at $t=1$. Therefore, taking residues at $t=1$, the left hand side of \eqref{loc 6.2} vanishes and we get that
\begin{align*}
    -[\oO_{M_+}\virt] + [\oO_{M_-}\virt] + \mathrm{res}_{t=1}  \frac{[\oO_F\virt]}{e(N\virt)}
\end{align*}
is zero, which is what we want.
\end{proof}

\medskip \subsection{Dual obstruction cone} We first recall the definition of the Jiang-Thomas dual obstruction cone \cite{JiangThomas}.

Let $X$ be a Deligne-Mumford stack equipped with a perfect obstruction theory $\phi \colon E \to \bL_X$ with obstruction sheaf $\fF = Ob_X =h^1(E^\vee)$.

\begin{defi}
The \emph{dual obstruction cone} of $X$ is defined by
\begin{align*}
    N = \Spec_X (\Sym \fF) \xrightarrow{\pi} X
\end{align*}
which is the functor that assigns to every morphism $\rho \colon U \to X$ the set $\Hom_U(\rho^* \fF, \oO_U)$.
\end{defi}

By standard perfect obstruction theory arguments, we can find an \'{e}tale cover $\lbrace X\lalp \to X \rbrace$, a smooth affine scheme $A\lalp$, a vector bundle $\vV\lalp$ on $A\lalp$ and a section $s\lalp \in H^0(A\lalp, \vV\lalp)$ such that $X\lalp$ is the zero locus of $s\lalp$ and the perfect obstruction theory on $X\lalp$ is given by the two-term complex
\begin{align} \label{local isomorphism induced from POT}
    E\lalp = E|_{X\lalp} \simeq [\vV\lalp^\vee|_{X\lalp} \xrightarrow{ds\lalp^\vee} \Omega_{A\lalp}|_{X\lalp} ]
\end{align}
together with the natural map to $\bL_{X\lalp} \simeq [I\lalp / I\lalp^2 \to \Omega_{A\lalp}|_{X\lalp}]$, where $I\lalp$ is the ideal sheaf of $X\lalp$ in $A\lalp$.

Let $N\lalp = N|_{X\lalp}$ and write $\pi\lalp \colon \vV\lalp^\vee |_{X\lalp} \to {X\lalp}$ for the projection. By definition, $N\lalp$ is the closed subscheme of $\vV\lalp^\vee|_{X\lalp}$ defined by the vanishing of the section
\begin{align*}
    ds\lalp^\vee \colon \vV\lalp^\vee |_{X\lalp} \lr \pi\lalp^* \Omega_{A\lalp} |_{X\lalp}.
\end{align*}

Let $x_1, \ldots , x_n$ be \'{e}tale coordinates on $A\lalp$ and $y_1,  \ldots , y_r$ coordinates on the fibers of the bundle $\vV\lalp^\vee$. Then $N\lalp$ is cut out by the equations
\begin{align} \label{loc 7.5}
    \lbrace s_i \rbrace_{1 \leq i \leq r}, \ \lbrace \sum_i y_i \frac{\partial s_i}{\partial x_j} \rbrace_{1 \leq j \leq n}
\end{align}
where $s_i = s_i(x_1, \cdots, x_n)$ are the coordinate functions of the section $s\lalp$.

Let $\ti{s}\lalp \colon \vV\lalp^\vee \to \bC$ be the function defined by the formula
\begin{align*}
    \ti{s}\lalp = \sum_i y_i s_i.
\end{align*}
The differential of $\ti{s}\lalp$ is then
\begin{align} \label{loc 7.7}
    d \ti{s}\lalp = \sum_i s_i dy_i + \sum_j \left( \sum_i y_i \frac{\partial s_i}{\partial x_j} \right) dx_j
\end{align}

Comparing \eqref{loc 7.7} with \eqref{loc 7.5} we see that $N\lalp$ is the d-critical locus of the function $\ti{s}\lalp$. It therefore admits a symmetric perfect obstruction theory (cf. \cite{BehFun}) $\psi \lalp \colon F\lalp \to \bL_{N\lalp}$ with
\begin{align} \label{local obstruction theory of Nlalp}
    F\lalp = [T_{\vV\lalp^\vee}|_{N\lalp} \xrightarrow{d(d\ti{s}\lalp)^\vee} \Omega_{\vV\lalp^\vee}|_{N\lalp}]
\end{align}
and obstruction sheaf $\Ob_{N\lalp}=h^1(F\lalp^\vee) = \Omega_{N\lalp}$.

\begin{thm} \label{theorem 7.4}
The \'{e}tale cover $\lbrace N\lalp \to N \rbrace$ and the (symmetric) perfect obstruction theories $\psi\lalp \colon F\lalp \to \bL_{N\lalp}$ form an almost perfect obstruction theory $\psi$ on $N$ with obstruction sheaf $\Ob_N= \Omega_N$.
\end{thm}

\begin{proof}
Let $\pi \colon N \to X$ denote the projection.
For any index $\alpha$, we have a commutative diagram
\begin{align} \label{loc 7.7a}
    \xymatrix{
    \pi^* E\lalp \ar@{=}[r] \ar[d] & [\pi^* \vV\lalp^\vee |_{X\lalp} \ar[d] \ar[r]^-{ds\lalp^\vee} & \pi^* \Omega_{A\lalp} |_{X\lalp}] \ar[d]\\
    F\lalp \ar[d] \ar@{=}[r] & [T_{\vV\lalp^\vee}|_{N\lalp} \ar[d] \ar[r]^-{d(d\ti{s}\lalp)^\vee} & \Omega_{\vV\lalp^\vee}|_{N\lalp}] \ar[d] \\
    \pi^* E\lalp^\vee[1] \ar@{=}[r] & [ \pi^* T_{A\lalp} |_{X\lalp} \ar[r]^-{(ds\lalp^\vee)^\vee} & \pi^* \vV\lalp |_{X\lalp}]
    }
\end{align}
giving an exact triangle $\pi^* E\lalp \to F\lalp \to \pi^* E\lalp^\vee[1]$ fitting in a commutative diagram
\begin{align} \label{loc 7.10}
    \xymatrix{
    \pi^* E\lalp \ar[r] \ar[d]_-{\pi^* \phi\lalp} & F\lalp \ar[d]_-{\psi \lalp} \ar[r]  & \pi^* E\lalp^\vee[1] \ar[r] \ar[d]^-{\theta\lalp} & \\
    \pi^* \bL_{X\lalp} \ar[r] & \bL_{N\lalp} \ar[r] & \bL_{N\lalp / X\lalp} \ar[r] &
    }
\end{align}
where $\pi^* \phi\lalp$ and $\theta\lalp$ are restrictions to $N\lalp$ of global arrows $\pi^* \phi \colon \pi^* E \to \pi^* \bL_X$ and $\theta \colon \pi^* E^\vee [1] \to \pi^* \fF \to \bL_{N / X}$ using the isomorphisms \eqref{local isomorphism induced from POT} and the definition of the dual obstruction cone $N$.
Since $\Ob_{N\lalp} = \Omega_{N\lalp} = \Omega_N |_{N\lalp}$ the obstruction sheaves glue to an obstruction sheaf $\Ob_N = \Omega_N$.

Let $X\lab =X\lalp \times_X X\lbet$ and $N\lab = N\lalp \times_N N\lbet = N \times_X X\lab$. Then \eqref{loc 7.10} together with the above discussion shows that there exist quasi-isomorphisms
\begin{align*}
    \eta\lab \colon F\lalp |_{N\lab} \lr F\lbet |_{N\lab}
\end{align*}
which are compatible with the morphisms $\psi\lalp|_{N\lab}$ and $\psi\lbet|_{N\lab}$ and respect the symmetry of the obstruction theories $F\lalp$ and $F\lbet$, thus inducing the gluing morphisms for the obstruction sheaf $\Omega_N$.
In particular, the axioms of an almost perfect obstruction theory are all satisfied.
\end{proof}

\begin{rmk}
The almost perfect obstruction theory of $N$ is symmetric in the sense of \cite{BehFun} with respect to the natural generalization of the definition in our context.
\end{rmk}

From now on, we assume that the perfect obstruction theory $E$ has a global resolution $E = [ E^{-1} \to E^0 ]$ where $E^{-1}, E^0$ are locally free sheaves on $X$.

The grading on $\Sym \fF$ determines a $T=\bC^*$-action on $N$, scaling the fibers of $N$ over $X$, whose fixed locus is precisely $X$.
Differentiating the $T$-action, we obtain the Euler vector field whose dual is a cosection
\begin{align} \label{loc 7.9}
    \sigma \colon \Omega_N \lr \oO_N
\end{align}
and whose vanishing locus is $X$ by \cite[Section~3]{JiangThomas}.

By construction, the almost perfect obstruction theory $\psi$ of Theorem~\ref{theorem 7.4} is $T$-equivariant and has obstruction sheaf $\Ob_N = \Omega_N$.
By \eqref{loc 7.7a}, the virtual normal bundle of $X$ inside $N$ is $E^\vee$ which by assumption admits the global resolution $[E_0 \to E_1]$ where as usual $E_i = (E^{-i})^\vee$ for $i=0,1$. Therefore Assumption \ref{56} holds and we may apply Theorem \ref{46} and Definition \ref{42} to obtain the following theorem, keeping in mind that the cosection $\sigma$ vanishes on the $T$-fixed locus $X$.

\begin{thm} \label{thm 7.6}
Let $X$ be a Deligne-Mumford stack with a perfect obstruction theory $\phi \colon E \to \bL_X$ with obstruction sheaf $\fF = \Ob_X$ such that $E$ admits a global resolution by locally free sheaves. Let $\psi$ denote the induced $T$-equivariant almost perfect obstruction theory of the dual obstruction cone $N = \Spec_X (\Sym F)$ with obstruction sheaf $\Ob_N = \Omega_N$ and cosection $\sigma$ as in \eqref{loc 7.9}.
Let $\iota \colon X \to N$ be the inclusion as the zero section.
Then the cosection localized virtual structure sheaf $[\oO_{N,\loc}\virt] \in K_0(X)$ and the virtual structure sheaves $[\oO_N\virt] \in K_0(N)$ and $[\oO_X\virt] \in K_0(X)$ are related by
\begin{align*}
    [\oO_N\virt] & = \iota_* [\oO_{N,\loc}\virt] \in K_0(N),\\
    [\oO_N\virt] & = \iota_* \frac{[\oO_X\virt]}{e(E^\vee)} \in K_0^T(N) \otimes_{\bQ[t,t^{-1}]} \bQ(t).
\end{align*}

In particular, when $X$ is proper, we may define $K$-theoretic invariants by taking regular and equivariant Euler characteristics respectively to obtain
\begin{align*}
    \chi \left( [\oO_{N,\loc}\virt] \right) \in \bQ, \quad \chi_t \left( \frac{[\oO_X\virt]}{e(E^\vee)}\right)   \in \bQ(t).
\end{align*}
\end{thm}

\appendix

\section{Deformation Invariance of Cosection Localized Virtual Structure Sheaf} \label{appendix}

Let $\phi$ be an almost perfect obstruction theory on $X \to S$, given by perfect obstruction theories $\phi\lalp \colon E\lalp \to \bL_{X\lalp / S}$ on an \'{e}tale cover $\lbrace X\lalp \to X \rbrace_{\alpha \in A}$ of $X$. Let $\sigma \colon \Ob_X \to \oO_X$ be a cosection. We follow the notation of \S\ref{Scos} throughout.

Suppose that we have a Cartesian diagram
\begin{align} \label{loc A.1}
\xymatrix{
    Y \ar[r]^-{u} \ar[d] & X \ar[d] \\
    Z \ar[r]_-{v} & W
    }
\end{align} 
where $Z,W$ are smooth varieties and $v$ is a regular embedding, so that we also have cartesian diagrams
\begin{align} \label{loc A.2}
    \xymatrix{
    Y\lalp \ar[r]^-{u\lalp} \ar[d] & X\lalp \ar[d] \\
    Y \ar[r]_-{u} & X.
    }
\end{align}

Suppose now that we have an almost perfect obstruction theory on $Y \to S$ given by perfect obstruction theories $\phi'\lalp \colon E\lalp' \to \bL_{Y\lalp/S}$ together with commutative diagrams
\begin{align} \label{compat obs th}
\xymatrix{
    E\lalp|_{Y\lalp} \ar[r]^-{g\lalp} \ar[d]_-{\phi\lalp|_{V\lalp}} & E\lalp' \ar[d]^-{\phi\lalp'} \ar[r] & N_{Z/W}^\vee |_{Y\lalp}[1] \ar@{=}[d] \ar[r] &\\
    \bL_{X\lalp/S}|_{Y\lalp} \ar[r] & \bL_{Y\lalp / S} \ar[r] & \bL_{Y\lalp/X\lalp} \ar[r] &
    }
\end{align}
of distinguished triangles which are compatible with the diagrams \eqref{compatibility data APOT} for $\phi$ and $\phi'$ such that we have exact sequences 
\begin{align} \label{loc A.4}
    N_{Z/W}|_{Y\lalp} \lr \Ob_{Y\lalp} \xrightarrow{h^1(g\lalp^\vee)} \Ob_{X\lalp}|_{Y\lalp} \lr 0.
\end{align}
that glue to a sequence 
\begin{align} \label{loc A.5}
    N_{Z/W}|_{Y} \lr \Ob_{Y} \lr \Ob_{X}|_{Y} \lr 0.
\end{align}

We obtain an induced cosection 
\begin{align*}
    \sigma' \colon \Ob_{Y} \lr \Ob_X|_Y \xrightarrow{\sigma|_Y} \oO_Y
\end{align*}
The Cartesian squares~\eqref{loc A.1} and \eqref{loc A.2} give rise to a diagram
\begin{align*}
    \xymatrix{
    Y\lalp(\sigma') \ar[r]^-{t\lalp} \ar[d] & X\lalp(\sigma) \ar[d] \\
    Y(\sigma') \ar[r]^-{t} \ar[d] & X(\sigma) \ar[d] \\
    Z \ar[r]_-{v} & W
    }
\end{align*}
with Cartesian squares.
\begin{thm}
$[\oO_{Y, \loc}\virt] = v^! [\oO_{X, \loc}\virt] \in K_0(Y(\sigma'))$.
\end{thm}

Here the Gysin map $v^! \colon K_0(X(\sigma)) \to K_0(Y(\sigma'))$ is defined by the formula
\begin{align}
    v^! [\aA] = [\oO_Z^W |_{X(\sigma)} \otimes_{\oO_{X(\sigma)}} \aA] \in K_0(Y(\sigma'))
\end{align}
where we fix $\oO_Z^W$ to be a finite locally free resolution of $v_* \oO_Z$. By \cite{yplee}, $v^!$ also equals the composition
\begin{align}
   K_0(X(\sigma)) \xrightarrow{\sigma_u} K_0(C_{Y(\sigma')/X(\sigma)}) \xrightarrow{0_{N_{Z/W}}^!} K_0(Y(\sigma'))
\end{align}
where $\sigma_u$ is specialization to the normal cone and $0_{N_{Z/W}}^!$ is the Gysin map induced from the Cartesian diagram
\begin{align*}
\xymatrix{
    Y(\sigma') \ar[d] \ar[r] & C_{Y(\sigma')/X(\sigma)} \ar[d] \\
    Z \ar[r] & N_{Z/W}.
    }
\end{align*}

\begin{proof}[Proof of Theorem A.1]
Let $\mM_X^\circ \to \bP^1$ be the deformation of $X$ to its intrinsic normal cone stack $\cC_X$ and $\wW = \mM_{Y \times \bP^1 / \mM_X^\circ}^\circ$ be the double deformation space given by the deformation of $Y \times \bP^1$ inside $\mM_X^\circ$ to its normal cone $\cC_{Y \times \bP^1 / \mM_X^\circ}$. 

As in the proof of Lemma~\ref{lemma 4.6}, we obtain
\begin{align}
   [\oO_{\cC_Y}] =  [\oO_{\cC_{Y/\cC_X}}] \in K_0(\cC_{Y\times \bP^1 / \mM_X^\circ}).
\end{align}

Since $\cC_{Y \times \bP^1/\mM_X^\circ}$ is a closed substack of $\nN_{Y \times \bP^1 / \mM_X^\circ}$, the equality holds in $K_0(\nN_{Y \times \bP^1 / \mM_X^\circ})$ as well.

Following \cite{KimKreschPant}, for each index $\alpha$ we consider the commutative diagram of distinguished triangles on $Y\lalp \times \bP^1$
\begin{align} \label{loc A.9}
\xymatrix{
E\lalp|_{Y\lalp}(-1) \ar[r]^-{\kappa\lalp} \ar[d] & E\lalp|_{Y\lalp} \oplus E\lalp' \ar[r] \ar[d] & c(\kappa\lalp) \ar[r] \ar[d] &\\
\bL_{X\lalp/S} |_{Y\lalp}(-1) \ar[r]_-{\lambda\lalp} & \bL_{X\lalp/S} |_{Y\lalp} \oplus \bL_{Y\lalp/S} \ar[r] & c(\lambda\lalp) \ar[r] &
}
\end{align}
where $\kappa\lalp = (T \cdot \id, U \cdot g\lalp)$ with $T,U$ coordinates on $\bP^1$.

Clearly $\lambda\lalp$ is the restriction to $V\lalp$ of a global morphism $\lambda$. By \cite{KimKreschPant}, we have that $h^1/h^0(c(\lambda)^\vee) = \nN_{Y \times \bP^1 / \mM_X^\circ}$.

By the properties of almost perfect obstruction theories and the compatibility diagrams \eqref{compat obs th}, the closed embeddings
$$h^1(c(\lambda\lalp)^\vee) \lr h^1(c(\kappa\lalp)^\vee)$$
glue to a closed embedding of sheaf stacks on $Y \times \bP^1$
$$\fn_{Y\times \bP^1 / \mM_X^\circ} \lr \kK.$$

The same argument works at the level of coarse moduli sheaves, where flatness stands for exactness of the pullback functor. Thus we deduce the equality
\begin{align} \label{loc 6.11}
    [\oO_{\fc_Y}] =  [\oO_{\fc_{Y/\cC_X}}] \in K_0(\kK).
\end{align}

The top row of \eqref{loc A.9} together with \eqref{loc A.4} and \eqref{loc A.5} give a commutative diagram
\begin{align*}
    \xymatrix{
    \kK \ar[r] \ar[d] & \Ob_X |_Y \oplus \Ob_{Y} \ar[r] \ar[d]^-{(\sigma|_Y, \sigma')} & \Ob_X|_Y (1) \ar[r] \ar[d]^-{\sigma|_Y(1)} & 0 \\
    \oO_{Y \times \bP^1}(-1) \ar[r] & \oO_{Y \times \bP^1} \oplus \oO_{Y \times \bP^1} \ar[r] & \oO_{Y \times \bP^1}(1) \ar[r] & 0
    }
\end{align*}
and therefore we obtain a (twisted) cosection $\bar{\sigma} \colon \kK \to L$, where $L$ is the line bundle $\oO_{Y \times \bP^1}(-1)$.

By \cite[Section~5]{KiemLiCosection}, $\fc_{Y \times \bP^1 / \mM_X^\circ}$ has reduced support in $\kK(\bar{\sigma})$.
The fiber of $\kK$ over $\lbrace 0 \rbrace \in \bP^1$ is $\Ob_{X}|_Y \oplus N_{Z/W}|_Y$ while the fiber over $\lbrace 1 \rbrace \in \bP^1$ is $\Ob_{Y}$. The cosection $\bar{\sigma}$ also restricts to the corresponding cosections over these two fibers. Therefore, we obtain by \eqref{loc 6.11}
\begin{align*}
    [\oO_{Y,\loc}\virt] = 0_{\Ob_{Y}, \sigma'}^! [\oO_{\fc_Y}] = 0_{\Ob_{X}|_Y \oplus N_{Z/W}|_Y, \sigma|_Y}^! [\oO_{\fc_{Y/\cC_X}}]
\end{align*}
Now, since the usual properties of Gysin maps hold by working on local charts of the corresponding sheaf stacks, we have
\begin{align*}
    0_{\Ob_{X}|_Y \oplus N_{Z/W}|_Y, \sigma|_Y}^! [\oO_{\fc_{Y/\cC_X}}] = 0_{\Ob_X|_Y, \sigma|_Y}^! 0_{N_{Z/W}|_Y}^! [\oO_{\fc_{Y/\cC_X}}] = 0_{\Ob_X |_Y, \sigma|_Y}^! v^! [\oO_{\fc_X}]
\end{align*}
By the next proposition, we have $0_{\Ob_X|_Y, \sigma|_Y}^! v^! = v^! 0_{\Ob_X, \sigma}^!$, which implies the desired equality.
\end{proof}

\begin{prop}
For any coherent sheaf $\aA$ on $\Ob_X$ supported on $\Ob_X(\sigma)$, $0_{\Ob_X|_Y, \sigma|_Y}^! v^! [\aA] = v^! 0_{\Ob_X, \sigma}^! [\aA]$.
\end{prop}

\begin{proof}
Since the pullback $v^!$ is given by tensoring with the resolution $\oO_Z^W$, the equality follows by combining the proofs of \cite[Lemma~5.6]{KiemLiKTheory} and \cite[Proposition~4.3]{KiemSavvas} together with the construction of the cosection localized Gysin map in \S\ref{Scos}.
\end{proof}

\bibliography{Master}
\bibliographystyle{alpha}

\end{document}